\documentclass[11pt,A4]{article}
\usepackage[applemac]{inputenc}
\usepackage[english]{babel}
\usepackage{amsmath,amsfonts,amssymb,amsthm,mathrsfs,bbm}
\usepackage[font=sf, labelfont={sf,bf}, margin=1cm]{caption}
\usepackage{graphicx,graphics}
\usepackage{epsfig}
\usepackage{latexsym}
 \usepackage{ae,aecompl}
\usepackage{pstricks}
\usepackage{enumerate}
\usepackage{xcolor}
\usepackage{pifont}
\usepackage[pdfpagemode=UseNone,bookmarksopen=false,colorlinks=true,urlcolor=blue,citecolor=blue,citebordercolor=blue,linkcolor=blue]{hyperref}
\usepackage[margin=2.5cm]{geometry}
\usepackage{comment}
\usepackage{mathpazo} 
\usepackage{eulervm}
\DeclareGraphicsExtensions{.jpg,.pdf}
\usepackage[normalem]{ulem}
\usepackage{enumerate,stmaryrd}
\usepackage{tikz,pgf}						
	\usetikzlibrary{arrows, patterns, calc}		
\usepackage{mathtools}					
\usepackage{microtype,etex}					
\usepackage{tabulary}
\usepackage{booktabs,longtable}

\usepackage{enumitem}

  \DeclareMathAlphabet{\mathpzc}{OT1}{pzc}{m}{it}

\newtheorem{thm}{Theorem}
\newtheorem{lem}[thm]{Lemma}
\newtheorem{prop}[thm]{Proposition}
\newtheorem{cor}[thm]{Corollary}

\theoremstyle{definition}

\newtheorem{rem}[thm]{Remark}

\renewcommand\Pr[1]{\mathbb{P}\left(#1\right)}

\newcommand\Prmu[1]{\mathbb{P}_{\mu}\left(#1\right)}
\newcommand\Prmuj[1]{\mathbb{P}_{\mu,j}\left(#1\right)}

\newcommand\Es[1]{\mathbb{E}\left[#1\right]}
\newcommand\Esmu[1]{\mathbb{E}_{\mu}\left[#1\right]}
\newcommand\Esmuj[1]{\mathbb{E}_{\mu,j}\left[#1\right]}

\newcommand \fl[1] {\left\lfloor #1 \right\rfloor}

\def \Card {\mathop{{\rm Card } }\nolimits}

\newcommand \Wlex {\W^{\mathsf{lex}}}
\newcommand \Wrev {\W^{\mathsf{rev}}}
\newcommand \Wbfs {\W^{\mathsf{bfs}}}

\def \P {\mathbb{P}}

\def \Pmu {\mathbb{P}_\mu}

\def \Pmuj {\mathbb{P}_{\mu,j}}

\def \N {\mathbb N}

\def \D {\mathbb D}

\def \R {\mathbb R}

\def \D {\mathbb D}

\def \Z {\mathbb Z}

\def \S {\mathcal{S}}
\def \W {\mathcal{W}}
\def \Sb {\overline{\mathcal{S}}}
\def \Znz {\Z/n\Z}
\def \I {\mathcal{I}}

\def \Tca {\mathcal{T}^{\ast}}

\def \V {\textbf{V}}

\def \bx {\textnormal{\textbf{x}}}

\def \exc {\textnormal{exc}}
\def \X {X^\textnormal{exc}}
\def \Xbr {X^\mathrm{br}}
\def \He {H^\textnormal{exc}}
\def \Xe {X^\textnormal{exc}}

\def \e {\epsilon}

\def \br {\textnormal{br}}
\def \GW {\textnormal{BGW}}

\def \t {\mathfrak{t}}

\def \| { \, | \,}

\def \br {\mathrm{br}}

\def \Z {\mathbb {Z}}

\long\def\symbolfootnote[#1]#2{\begingroup%
\def\thefootnote{\fnsymbol{footnote}}\footnote[#1]{#2}\endgroup}

\title{  \vspace {-2cm}\textbf{Sub-exponential tail bounds for \\conditioned stable Bienaymé--Galton--Watson trees}}
\date{}

\DeclareSymbolFont{extraup}{U}{zavm}{m}{n}
\DeclareMathSymbol{\varheart}{\mathalpha}{extraup}{86}
\DeclareMathSymbol{\vardiamond}{\mathalpha}{extraup}{87}

\makeatletter
\renewcommand*{\@fnsymbol}[1]{\ensuremath{\ifcase#1\or  \spadesuit \or \varheart\or \vardiamond \or \clubsuit \or
   \mathsection\or \mathparagraph\or \|\or **\or \dagger\dagger
   \or \ddagger\ddagger \else\@ctrerr\fi}}
\makeatother

\author{Igor Kortchemski\thanks{CNRS \& CMAP, \'Ecole polytechnique, Palaiseau, France. \hfill  \texttt{igor.kortchemski@normalesup.org}} 
\ 
}

\begin{document}

\maketitle

\let\thefootnote\relax\footnotetext{  \\ \emph{The  author acknowledges partial support from Agence Nationale de la Recherche, grant number ANR-14-CE25-0014 (ANR GRAAL), and from the City of Paris, grant “Emergences Paris 2013, Combinatoire à Paris”.}
\\ 

\emph{MSC2010 subject classifications}. Primary 60J80,05C05,05C07; secondary 60F05,60G52. }
 
\vspace {-0.5cm}

\begin{abstract} 
We establish uniform sub-exponential tail bounds for the width, height and maximal outdegree of critical Bienaymé--Galton--Watson trees conditioned on having a large fixed size, whose offspring distribution  belongs to the domain of attraction of a stable law. This extends results obtained for the height and width by Addario-Berry, Devroye \& Janson in the finite variance case.
\end{abstract}

\section{Introduction}

We are interested in the geometric structure of large Bienaymé--Galton--Watson trees, which are an important well-studied class of random trees in probability theory. They arise as building blocks of many different models of random graphs, such as Erd\H{o}s--R\'enyi random graphs or random maps, and appear in combinatorics under the term of simply-generated trees. 
Addario-Berry, Devroye \& Janson \cite{ADS13} established sub-Gaussian tail bounds for the width and height of critical finite-variance Bienaymé--Galton--Watson trees conditioned on having a fixed size, these bounds being uniform in the size. 
Such uniform bounds are often challenging to prove, and are useful tools to establish scaling limits for various families of random graphs, see for example \cite{Ray15,CHK14,BM14,PSK14,Stu15}.
Our goal is to obtain similar sub-exponential bounds for the width, height and also maximal outdegree of critical Bienaymé--Galton--Watson trees conditioned on having a fixed size, but whose offspring distribution belongs to the domain of attraction of a stable law and may have infinite variance. 
We believe that our results should find applications in the study of scaling limits of random graphs exhibiting heavy tail phenomena (see e.g.~\cite{BHS15}).
 	
	Bienaymé--Galton--Watson trees belonging to domains of attraction of stable laws have recently appeared in a number of two-dimensional statistical physics models, in connection with scaling limits of random maps with large faces \cite{LGM11}, scaling limits of critical site-percolation clusters on infinite random triangulations \cite{CK13+} via looptrees \cite{CK14} and Liouville quantum gravity \cite{DMS14}.

\paragraph{Stable Bienaymé--Galton--Watson trees.} We consider  critical offspring distributions belonging to domains of attraction of stable laws. Specifically, we fix a parameter $\alpha \in (1,2]$ and let $\mu= (\mu(j); j \geq 0)$ be a probability distribution on the nonnegative integers satisfying the following two conditions:
\begin{enumerate}
\item[(i)] $\mu$ is critical, meaning that $ \displaystyle \sum_{j=0}^\infty j \mu(j) = 1$.
\item[(ii)]  $\mu
$ belongs to the domain of attraction of a stable law of index $\alpha
\in (1,2]$.
\end{enumerate}
By \cite[Theorem XVII.5.2]{Fel71}, assertion (ii) means that if $X$ is a random variable with distribution $\mu$, then $\textrm{Var}(X \cdot \mathbb{1}_{X \leq n})= n^{2-\alpha} L(n)$, where $ L: \R_+ \rightarrow \R_+$ is a function such that $\lim_{x
\rightarrow \infty} L(tx)/ L(x)=1$ for all $t>0$ (such a function is
called slowly varying).  Equivalently, either the variance of $\mu$ is finite, or $n^{\alpha}\mu([n,\infty))$  is another slowly varying function (see \cite[Eq.~(5.16) and Theorem 2 in Sec.~XVII.5]{Fel71}).

In addition, we always implicitly suppose that $ \mu(0)+\mu(1)<1$ to avoid degenerate cases, and always assume that $ \mu$ is aperiodic, in the sense that the additive subgroup of the integers $\mathbb{Z}$ spanned by $\{j;
\, \mu(j) \neq 0 \}$ is $\mathbb{Z}$. We let $ \Pmu$ denote the law of a (plane, rooted) Bienaymé--Galton--Watson tree with offspring distribution $\mu$, and $|\tau|$ be total number of vertices, or size, of a tree $\tau$. For every $n \geq 1$ such that $ \Prmu { |\tau|=n}>0$, $ \t_n$ will denote a $ \GW_ \mu$  tree conditioned on having $n$ vertices. The aperiodicity of $ \mu$ guarantees that $ \Prmu { |\tau|=n}>0$ for every $n$ sufficiently large (our results carry out to the periodic case with mild modifications, and we only focus on the aperiodic case for simplicity).

\paragraph{Asymptotic behavior of stable Bienaymé--Galton--Watson trees.} The asymptotic behavior of $\t_{n}$ is well understood, in particular through scaling limits of different functions coding $\t_{n}$. Specifically, if $u(0), u(1), \ldots, u(n-1)$ denote the vertices of $\t_{n}$ listed in lexicographical order (see Sec.~\ref{sec:GW} for precise definitions), define the height function $(\mathcal{H}_{i}(\t_{n}); 0 \leq i \leq n)$ by letting $\mathcal{H}_{i}(\t_{n})$ be the generation of $u(i)$ in $\t_{n}$ for $0 \leq i \leq n-1$ and setting $ \mathcal{H}_{n}(\t_{n})=0$ by convention. Define also the {\L}ukasiewicz path $(\mathcal{W}_{i}(\t_{n}); 0 \leq i \leq n)$ of $\t_{n}$ by setting $\mathcal{W}_{0}(\t_{n})=0$ and, for $0 \leq i \leq n-1$, $\mathcal{W}_{i+1}(\t_{n})-\mathcal{W}_{i}(\t_{n})+1$ to be the outdegree (i.e.~number of children) of $u(i)$.  Let $(B_{n})_{n \geq 1}$ be an increasing sequence such that if $(X_{i})_{i \geq 1}$ is a sequence of i.i.d.~random variables with distribution $\mu$, $(X_{1}+X_{2}+ \cdots+X_{n}-n)/B_{n}$ converges in distribution as $n \rightarrow  \infty$ to a random variable $Y_{\alpha}$ with Laplace exponent given by $\Es{ \exp(- \lambda Y_{\alpha})}= \exp(\lambda ^ \alpha)$ for every $\lambda>0$ (\cite[Sec.~XVII.5]{Fel71} garanties its existence). 
Duquesne \cite{Du03} (see also \cite {Kor13}) showed that the convergence
\begin{equation}
\label{eq:D}\left( \frac{1}{B_n} \cdot  \W_ { \fl {nt}}( \t_n),  \frac{n}{B_{n}} \cdot \mathcal{H}_{ \fl{nt}}(\t_{n}) \right) _{0 \leq t \leq 1}  \quad\mathop{\longrightarrow}^ {(d)}_{n \rightarrow \infty} \quad \left( \Xe_{t},\He_{t} \right)_{0 \leq t \leq 1}
\end{equation}
holds in distribution in $\D([0,1], \R)^{2}$, where $\D([0,1], \R)$ is the space of real-valued càdlàg functions on $[0,1]$ equipped with the Skorokhod $J_{1}$ topology, where $\Xe$ is the normalized excursion of a spectrally positive strictly stable Lévy process of index $\alpha$ and $\He$ its associated continuous height function (we refer to  \cite{Du03} for their construction as we will not use them) which codes the $\alpha$-stable Lévy tree introduced by Le Gall \& Le Jan \cite{LGLJ98}.  In the particular case $\alpha=2$, we have  $(\Xe,\He)= (\sqrt{2}\cdot \mathbbm{e}, \sqrt{2} \cdot\mathbbm{e})$, where $\mathbbm{e}$ is the normalized Brownian excursion. The scaling factor $B_{n}$ is of order $n^{1/\alpha}$ (more precisely, $B_{n}/n^{1/\alpha}$ is slowly varying), and  one may take $B_{n}= \sigma  \sqrt{n/2}$ when $\mu$ has finite variance $\sigma^{2}$.

\paragraph{Uniform bounds on the width and height of stable Bienaymé--Galton--Watson trees.}

If $\tau$ is a (plane, rooted) tree and $k \geq 0$, we denote by $Z_{k}(\tau)$ the number of vertices of $\tau$ at generation $k$, and let
$$W(\tau)= \sup \{Z_{k}(\tau);  \ k \geq 0\}, \qquad  H(\tau)= \sup \{k; \ Z_{k}(\tau)>0\} $$
be respectively the width and height of $\tau$. We also denote by $\Delta(\tau)$ the maximum outdegree of $\tau$. To simplify notation, we will sometimes write $Z_{k},W,H$ instead of respectively $Z_{k}(\tau), W(\tau)$ and $H(\tau)$. Let $\Delta^*( \X)= \sup_ {0 < t \leq 1} ( \X_t- \X_ {t-})$ be the maximum jump of $\Xe$. Since the maximum jump of $ \mathcal{W}(\t_{n})$ is equal to $\Delta(\t_{n})-1$ and the largest jump is a continuous functional on $\D([0,1], \R)$ (see e.g.~\cite[Prop.~2.4 in Chap.~VI]{JS03}; we emphasize that we are always working with the $J_{1}$ topology), \eqref{eq:D} immediately implies that
 $$ \left( \frac{1}{B_{n}} \cdot \Delta(\t_{n}), \frac{n}{B_{n}} \cdot  {H}(\t_{n}) \right)   \quad \mathop{\longrightarrow}^{(d)}_{n \rightarrow \infty} \quad  (\Delta^*( \X), \sup \He).$$
It is also plausible that this convergence holds jointly with that of $ W(\t_{n})/B_{n}$ to a positive random variable, see \cite{Ker} (unfortunately, we have not managed to found a published reference of this fact). As a consequence, for every $u \geq 0$, the quantities $ \Pr{H(\t_{n}) \geq  un /B_{n}}$, $ \Pr{W(\t_{n}) \geq u B_{n}}$ and $ \Pr{\Delta(\t_{n}) \geq u B_{n}}$ should converge as $n \rightarrow \infty$ to functions of $u$ that tend to $0$ as $u \rightarrow \infty$. It is therefore natural to ask if it is possible to bound  $ \Pr{H(\t_{n}) \geq  un /B_{n}}$, $ \Pr{W(\t_{n}) \geq u B_{n}}$ and $ \Pr{\Delta(\t_{n}) \geq u B_{n}}$ by functions of $u$ which do \emph{not} depend on $n$.

In the case where $\mu$ is critical and has finite positive variance, such bounds have been established by Addario-Berry, Devroye \& Janson \cite[Theorems 1.1 and 1.2]{ADS13}, who show the existence of constants $C_{1},c_{1}>0$ (depending only on $\mu$) such that the inequalities
$$ \Pr{H(\t_{n}) \geq u \sqrt{n}} \leq C_{1} e^{-c_{1} u^{2}}, \qquad \Pr{W(\t_{n}) \geq u \sqrt{n}} \leq C_{1} e^{-c_{1} u^{2}}$$
hold for every $n \geq 1$ and $u \geq 0$. Addario-Berry \cite{Add12} establishes similar bounds for uniform random trees with a given outdegree sequence satisfying a ``finite variance'' type condition. When $\mu$ is critical and belongs to the domain of attraction of a stable law,  Haas \& Miermont \cite[Lemma 33]{HM12} show that for every $p>0$, there exists a constant $C_{p}>0$ such that
$$ \Pr{H(\t_{n}) \geq  \frac{u n}{B_{n}}} \leq  \frac{C_{p}}{u^{p}}$$
for every $n \geq 1$ and $u \geq 1$ (this reference actually treats the more general case of so-called Markov-branching trees).

We are now in position to state our main results. Recall that $\t_{n}$ denotes a Bienaymé--Galton--Watson tree with a critical offspring distribution in the domain of attraction of a stable law of index $\alpha \in (1,2]$, conditioned on having $n$ vertices.

\begin{thm}[Bounds for the width] \label{thm:width}For every $ \gamma \in (0, \alpha/ ( \alpha-1))$, there exist positive constants $C_1,C_2>0$ such that for every $ u \geq 0$ and every $n \geq 1$:
$$ \Pr { W( \t_n) \geq u B_n} \leq C_1 \exp (- C_2 u^ { \gamma}).$$
\end{thm}

The exponent $\alpha/(\alpha-1)$ is optimal. We will see this by explicitly calculating the tail of the supremum of the stable bridge (Theorem \ref{thm:Xbr}) and evaluating its asymptotic behavior (Corollary \ref{cor:Xbr}), which are results of independent interest. See also \cite{VVF07} for a study of the width of (non-conditioned) stable Bienaymé--Galton--Watson trees.

\begin{thm}[Bounds for the height]\label{thm:height}For every $ \delta \in (0, \alpha)$, there exist positive constants $C_1,C_2>0$ such that for every $ u \geq 0$ and every $n \geq 1$:
$$ \Pr { H( \t_n) \geq  u \cdot  \frac{n}{B_{n}}} \leq C_1 \exp (- C_2 u^ \delta).$$
\end{thm}
Here, the exponent $\alpha$ is optimal. Indeed, for every $u \geq 0$, we have $ \Pr { H( \t_n) \geq  un/B_{n} } \rightarrow  \Pr{ \He \geq u}$ as $n \rightarrow  \infty$ by \eqref{eq:D}, and  it is shown in \cite[Theorem 1.5]{DW15} that
\begin{equation}
\label{eq:Hinfini} \Pr{ \sup \He >u}  \quad \mathop{\sim}_{u \rightarrow \infty} \quad \beta \cdot u^{1+ \frac{\alpha}{2}} e^{- (\alpha-1)^{ \frac{1}{\alpha-1}} u^{\alpha}},
\end{equation}
where $\beta>0$ is a positive constant depending only on $\alpha$.

As noted in \cite{ADS13}, since $H(\t_{n}) W(\t_{n}) \geq n-1$, the previous results also yield, for every  $ \gamma \in (0, \alpha/ ( \alpha-1))$ and $ \delta \in (0, \alpha)$, the existence of constants $C_{1},C_{2}>0$ such that
$$ \Pr { W( \t_n) \leq \frac{B_{n}}{u}} \leq C_1 \exp \left( - C_2 {u^ { \delta}} \right)  \qquad  \textrm{and} \qquad  \Pr { H( \t_n) \leq   \frac{1}{u} \cdot  \frac{n}{B_{n}}} \leq C_1 \exp  \left( - C_2 {u^ \gamma} \right)$$
for every $u > 0$ and $n \geq 1$. We believe that the exponent $\alpha$ is optimal for the first inequality. Also, the exponent $\alpha/(\alpha-1)$ is optimal for the second inequality, since, by \cite[Theorem 1.8]{DW15},
\begin{equation}
\label{eq:Hzero}\Pr{ \sup \He < \frac{1}{u}} \quad \mathop{\sim}_{u \rightarrow \infty} \quad \lambda \cdot \frac{1}{u^{\alpha+2+ \frac{1}{\alpha-1}}} \exp \left(  - \left(  \frac{\pi/\alpha}{\sin(\pi/\alpha)}\right) ^{ \frac{\alpha}{\alpha-1}} \cdot {u^{ \frac{\alpha}{\alpha-1}}} \right),
\end{equation} 
where $\lambda>0$ is a positive constant depending only on $\alpha$.

\paragraph{Application to the maximal outdegree of stable Bienaymé--Galton--Watson forests.} By using Theorem \ref{thm:width}, we establish the following result.

\begin{thm}[Bounds for the maximal outdegree]\label{thm:degree}For every $ \delta \in (0, \alpha/ ( \alpha-1))$, there exist positive constants $C_1,C_2>0$ such that for every $ u \geq 0$ and every $n \geq 1$:
$$ \Pr { \Delta( \t_n) \geq u B_n} \leq C_1 \exp (- C_2 u^ { \delta}).$$
\end{thm}

Here we believe that the exponent $\alpha/(\alpha-1)$ is also optimal. See \cite{Pak98,RaYa99,Ber11,Ber13} for results concerning the maximal degree of stable Bienaymé--Galton--Watson trees conditioned on non-extinction at high generation or for the maximal degree of forests.

In addition, we establish the following bounds, which are sharper when $\mu$ has finite variance and which also apply to forests of Bienaymé--Galton--Watson trees. For $j \geq 1$, denote by $\Pmuj$ the law of $j$ independent $\GW_{\mu}$ trees.
\begin {thm}[Bounds for the maximal outdegree of a forest]\label {thm:degree2}For every $M>0$, there exist  constants $C_{1},C_{2} >0$ such that, for every $n,k \geq 1$,
\begin{equation}
\label{eq:bd1}  \sup_{1 \leq j \leq M B_{n} }\Prmuj{ \Delta(F) \leq  k \ \big| \ |F|=n} \leq C_{1} \exp  \left( - C_{2} n \mu([k+1,\infty)) \right)\end{equation}
and
\begin{equation}
\label{eq:bd2} \sup_{1 \leq j \leq M B_{n}} \Prmuj{ \Delta(F) \geq   k \ \big| \ |F|=n}  \leq C_{1}  n \mu([k,\infty))
\end{equation}
\end {thm}
As an application of these bounds, we obtain concentration inequalities for the maximum outdegree of a large uniform non-crossing tree (Theorem \ref{thm:nc} below), improving a result by Deutsch \& Noy \cite {DN02}.

\paragraph{Sizes of generations in  stable Bienaymé--Galton--Watson trees.}
Theorem \ref{thm:height} estimates the probability that $ \t_{n}$ has a large height, namely at least $un/B_{n}$. One may then wonder: what is the size of the generation at level $un/B_{n}$, on the event that $\t_{n}$ has height at least $un/B_{n}$? In this direction, we establish the following bounds.
\begin{thm}\label{thm:gen}For every $ \gamma \in (0, \alpha-1)$ and $\eta>0$, there exists a positive constant $C_1>0$ such that for every $ u \geq \eta$,  $v \geq 0$ and every $n \geq 1$:
$$\Pr{0<Z_{u \frac{n}{B_{n}}}(\t_n)< v B_{n}} \leq C_{1} v^{\gamma}.$$\end{thm}

When $u$ takes values in a compact subset of $(0,\infty)$, the exponent $\alpha-1$ is optimal (see Remark \ref{rem:CK}). In proving this, we  crucially rely on the following uniform estimate, concerning the size $Z^{\ast}_{n}$ at generation $n$ of a $\GW_{\mu}$ tree conditioned to survive (see Sec.~\ref{sec:obs} for a definition), and which is of independent interest.

\begin {prop}\label{prop:Z}Set $p_{n}= \Prmu{H \geq n}$. For every $ \beta \in (0,\alpha)$, there exists a constant $C >0$ such that
\begin{equation}
\label{eq:Z}\textrm {for every } n \geq 1, x \geq 0, \qquad \Pr {p_{n}Z^{\ast}_{n} \leq x} \leq C \cdot x ^ {\beta}.
\end{equation}
\end {prop}

Let us mention that using a different approach, Croydon \& Kumagai \cite[Proposition 2.6]{CK08} show the weaker result that for every $ \beta \in (0,\alpha-1)$, there exists a constant $C >0$ such that \eqref{eq:Z} holds. However, in our case, it is important to be able to choose $\beta> \alpha-1$. The exponent $\alpha$ is optimal, since by \cite[Theorem 4]{Pak76}, $p_{n}Z_{n}^{\ast}$ converges in distribution to a random variable $Z^{\ast}$ with Laplace transform given by
$$\Es{e^{-\lambda Z^{\ast}}}= \frac{1}{ \left( 1+ \lambda^{\alpha-1} \right)^{ \frac{\alpha}{\alpha-1}}}, \qquad \lambda \geq 0,$$
which shows that for every $\epsilon>0$, there exists $C>0$ such that $ \Pr{Z^{*} \leq x} \geq  C x^{\alpha+\epsilon}$ for every $0 \leq x \leq 1$. It would be interesting to know whether \eqref{eq:Z} holds for $\beta=\alpha$.  

\begin{cor}\label{cor:gen}
\begin{enumerate}
\item[(i)]  For every $ \gamma \in (0, \alpha/ ( \alpha-1))$ and $ \delta \in (0, \alpha)$, there exist positive constants $C_1,C_2>0$ such that for every $ u,v \geq 0$ and every $n \geq 1$:
$$ \Pr{Z_{u \frac{n}{B_{n}}}(\t_n)> v B_{n}} \leq C_{1} \exp(-C_{2}(u^{\delta}+v^{\gamma})).$$
\item[(ii)] For every $ \gamma \in (0, (\alpha-1)/2)$ and $ \delta \in (0, \alpha)$, there exist positive constants $C_1,C_2>0$ such that  for every $ u \geq \eta,v \geq 0$ and every $n \geq 1$:
$$ \Pr{0<Z_{u \frac{n}{B_{n}}}(\t_n)< v B_{n}} \leq C_{1} v^{\gamma}\exp(-C_{2} u^{\delta})$$
\end{enumerate}
\end{cor}

The second assertion gives a better bound than Theorem \ref{thm:gen} when $u$ is large, but we believe that the exponent $(\alpha-1)/2$ is not optimal here.

\paragraph{Techniques.} We now comment on the main tools involved in the proof of the bounds for the width and height, and in particular on their connections with   \cite{ADS13}. The main tool, as in \cite{ADS13}, is the coding of conditioned Bienaymé--Galton--Watson trees by their {\L}ukasiewicz paths, which are, roughly speaking, nonnegative spectrally positive random walks conditioned on a late return to $0$. In order to establish the bounds on the width and maximal outdegrees, we establish bounds on the supremum of such walks, following the idea of  \cite{ADS13} that since these walks are spectrally positive, reaching high values and then returning to $0$ has a sub-exponential cost. However, the implementation of this idea is different, since \cite{ADS13} uses a bound that is only known to hold in the finite variance case (see the discussion after the statement of Theorem \ref{thm:borneW}). The starting observation for proving the bounds for the height is the same as in \cite{ADS13}: if a  conditioned Bienaymé--Galton--Watson tree has a large height, then vertices at high generations will have a lot of children branching off their ancestral line to the root, resulting in a large value for the supremum of the {\L}ukasiewicz path, which we already know to have an sub-exponential cost. However, a major difference is that the proof of \cite{ADS13} crucially uses the fact that the width and height of $\t_{n}$ are of the same order $\sqrt{n}$, which breaks down when $\mu$ has infinite variance, and the proof thus requires new ideas.

\paragraph{Acknowledgments.} The author is grateful to Louigi Addario-Berry and to Svante Janson for stimulating discussions, as well as to an anonymous referee for her or his extremely careful reading and many comments that greatly improved the paper, and would like to thank the Newton Institute, where this work was finalized, for hospitality.

\section{Estimates for random walks}
\label{sec:RW}

Recall that $\mu$ is a critical offspring distribution belonging to the domain of attraction of a stable law of index $\alpha \in (1,2]$. Let $(W_n)_{n \geq 0}$ be a random walk with starting point $W_0=0$ and jump distribution given by $ \Pr{W_{1}=i}= \mu(i+1)$ for $i \geq -1$.  Observe that  $\Es{W_{1}}=0$ since $\mu$ is critical. 

In this section, we study statistics of the random walk $(W_{n})_{n \geq 0}$ under different types of conditioning. They will play an important role since we will later see that Bienaymé--Galton--Watson related are coded by such walks. It may be useful to refer to Table \ref{tab:sec2} to keep track of the main notation of this section.

\begin{table}[htbp]\caption{Table of the main notation and symbols appearing in Section \ref{sec:RW}.}
\centering
\begin{tabular}{c c p{12cm} }
\toprule
$\mu$ &  & Critical offspring distribution on $\Z_{+}$ in the domain of attraction of a stable law of index $\alpha \in (1,2]$.\\
$L(n)$ &  & Slowly varying function such that $\textrm{Var}(X \cdot \mathbb{1}_{X \leq n})= n^{2-\alpha} L(n)$, with $X$ distributed as $\mu$.\\
$(W_{n}; {n \geq 0)}$ & & Random walk  with $W_0=0$ and jump distribution $ \Pr{W_{1}=i}= \mu(i+1)$ for $i \geq -1$. \\
$\zeta_{j}$ && is  $\inf \{n \geq 1: W_{n}=-j\}$.\\
$Y_{\alpha}$ &  & Stable random variable with Laplace exponent given by $\Es{ \exp(- \lambda Y_{\alpha})}= \exp(\lambda ^ \alpha)$ for  $\lambda>0$. \\
 $d_{\alpha}(x)$ && Density of $Y_{\alpha}$ at $x \in \R$.\\
$(B_{n})_{n \geq 1}$ && Increasing sequence such that $W_{n}/B_{n} \rightarrow Y_{\alpha}$ in distribution.\\
$\X$ && Normalized excursion of a spectrally positive strictly stable Lévy process of index $\alpha$. \\
$\Xbr$ && Bridge of a spectrally positive strictly stable Lévy process of index $\alpha$.\\
\bottomrule
\end{tabular}
\label{tab:sec2}
\end{table}

\subsection{Large deviations for left-continuous random walks}

Recall  from the Introduction that $(B_{n})_{n \geq 1}$ is an increasing sequence such that $W_{n}/B_{n}$ converges in distribution as $n \rightarrow  \infty$ to the random variable $Y_{\alpha}$ with Laplace exponent given by $\Es{ \exp(- \lambda Y_{\alpha})}= \exp(\lambda ^ \alpha)$ for every $\lambda>0$. We let $d_{\alpha}(x)$ denote the density of $Y_{\alpha}$ at $x \in \R$. Note that $d_{2}(x)=e^{-x^{2}/4}/\sqrt{4 \pi}$ is the density of a centered Gaussian distribution with variance $2$, and that $d_{\alpha}(0)=|\Gamma(-1/\alpha)|^{-1}$ (see \cite[Lemma XVII.6.1]{Fel71}).

 In addition, \begin{equation}
\label{eq:cvB} \frac{n L(B_{n})}{B_{n}^{\alpha}}  \quad \mathop{\longrightarrow}_{n \rightarrow \infty} \quad \frac{1}{(2-\alpha)\Gamma(-\alpha)},
\end{equation}
where we recall that $L$ is the slowly varying function such that  $\textrm{Var}(X \cdot \mathbb{1}_{X \leq n})= n^{2-\alpha} L(n)$ with $X$  a random variable distributed according to $\mu$ (by continuity, the quantity $((2-\alpha)\Gamma(-\alpha))^{-1}$ is interpreted as equal to $2$ for $\alpha=2$).  Indeed, in the notation of \cite[Sec.~4.5.1]{Whi02},  $Y_{\alpha}$  is the stable random variable $S_{\alpha}(|\cos(\pi \alpha /2)|^{1/\alpha},1,0)$   by \cite[Eq.~(5.17)]{Whi02}. First assume that $ 1 < \alpha <2$.  It follows from  \cite[Eq.~(5.16)]{Fel71} that $\Pr{|W_{1}| \geq n}  \sim \Pr{W_{1} \geq n} \sim \frac{2-\alpha}{\alpha} L(n) n^{-\alpha}$ as $n \rightarrow \infty$. Thus, by \cite[Theorem 4.5.1]{Whi02},
$$  \frac{n L(B_{n})}{B_{n}^{\alpha}}  \quad \mathop{\longrightarrow}_{n \rightarrow \infty} \quad \  \frac{\alpha}{2-\alpha} \cdot \frac{1-\alpha}{\Gamma(2-\alpha)} = \frac{1}{(2-\alpha)\Gamma(-\alpha)}.$$
If $\alpha=2$, we have $\Es{W_{1}^{2} \mathbbm{1}_{|W_{1}| \leq n}} \sim L(n) n^{2-\alpha}$ as $n \rightarrow \infty$, since $ \sum_{i=0}^{n} i^{2} \mu(i)-1 \sim L(n) n^{2-\alpha}$  by the definition of $L$. It then follows from
\cite[Theorem XVII.5.3]{Fel71} that $nL(B_{n})/B_{n}^{2} \rightarrow 2$.

The local limit theorem  \cite[Theorem 4.2.1]{IL71} shows that
\begin{equation}
\label{eq:LL}\sup_{k \in \Z} \left| B_{n} \Pr{W_{n}=k}- d_{\alpha} \left(  \frac{k}{B_{n}} \right)  \right|  \quad \mathop{\longrightarrow}_{n \rightarrow \infty} \quad 0.
\end{equation}

 By the so-called representation theorem (see e.g.  \cite[Theorem 1.3.1]{BGT87}), we can write $$ L(x)=c(x) \exp \left( \int_1^x \frac{\eta(u)}{u} du\right), \qquad x \geq 0,$$
where $c$ is a non-negative measurable function having a finite positive
limit at infinity and $\eta$ is a measurable function tending to $0$
at infinity. It easily follows that if $(x_n, y_n)_ { n\geq 1}$ are two sequences tending to infinity, for every $ \e>0$, there exists a constant $C>1$ such that for every integer $n$ sufficiently large:\begin{equation}
\label{eq:sv2} C^ {-1} \left( \frac{\max (x_n,y_n)}{ \min (x_n,y_n)} \right)^ {-\e}
 \leq \frac {L(y_n)}{L(x_n)} \leq C  \left( \frac{\max (x_n,y_n)}{ \min (x_n,y_n)} \right)^\e,
\end{equation}
Similar bounds hold with $L(n)$ replaced by $B_{n}/n^{1/\alpha}$, since the latter quantity is slowly varying.
In the literature, these inequalities are known as the Potter bounds.

We shall establish the following estimate.

\begin{prop} \label{prop:devRW}For every $ \delta \in (0, \alpha/ ( \alpha-1))$ there exists $C_1,C_2>0$ such that for every $ u \geq 0$ and every $n \geq 1$:
$$ \Pr { \min_ {0 \leq i \leq n} W_i \leq - u B_n} \leq C_1 \exp (- C_2 u^ \delta).$$
\end {prop}

\begin{proof} Here $C$ will stand for a positive constant which may vary from expression to expression (but that is independent of $u$ and $n$). Note that $ \Pr { \min_ {0 \leq i \leq n} W_i  \leq - u B_n} =0$ if $ u B_n > n$, so that we can suppose without loss of generality that $ 1 \leq u \leq n/B_n$.  Write, for $h>0$:
\begin{eqnarray}\Pr { \min_ {0 \leq i \leq  n} W_i \leq - u B_n} \notag 
&=&  \Pr  {  \max_ {0 \leq i \leq  n}  e^ { h (-W_i) } \geq  e^{  h u B_n}} \notag \\
& \leq &   {e^{  - h u B_n}} \Es {  e^ { -h W_n }}= {e^{  - h u B_n}} \Es {  e^ { -h W_1 }} ^n \label {eq:m3}
\end{eqnarray}
where we have used Doob's maximal inequality with the submartingale $ \left( e^ { h (-W_n)} ; n \geq 0\right)$ for the inequality.  Fix $ \eta \in (0, { 1}/({ \alpha-1}))$ and note that $ \eta  \alpha< 1+ \eta$. We shall apply the inequality  \eqref{eq:m3} with $ h= h_{n}(u)=u^ \eta / B_n$.  Observe that by the Potter bounds, $u^ \eta / B_n \rightarrow 0$  as $n \rightarrow \infty$, uniformly in $1 \leq u \leq n/B_n$. Therefore, by the estimate \eqref{eq:equivalents} of the Appendix, for every $n \geq 1$ and $1 \leq  u \leq n/B_n$,
$$\Es {  e^ { - \frac{u^ \eta}{B_n} W_1 }} ^n  \leq \exp \left(  C  n L \left( \frac{B_n}{u ^ \eta} \right) \frac{u ^ { \eta \alpha}}{ B_n^ \alpha}\right) $$
Now choose $ \epsilon>0$ such that $1 + \eta > \eta ( \alpha+ \epsilon)$.  For every $n$ sufficiently large and $u \geq 1$, we have $L(B_n/u^ \eta)\leq C u^ { \epsilon  \eta } L(B_n)$  by the Potter bounds, so that $$\exp \left(  C  n L \left( \frac{B_n}{u ^ \eta} \right) \frac{u ^ { \eta \alpha}}{ B_n^ \alpha}\right)  
\leq \exp\left(  C   \frac{n L \left( {B_n}\right)}{ B_n^ \alpha}  u ^ { \eta (\alpha+ \epsilon')}\right).$$
By \eqref{eq:cvB}, ${n L \left( {B_n}\right)}/{ B_n^ \alpha}$ is bounded as $n$ varies. It follows that for $n$ sufficiently large and $1 \leq u \leq   n/B_{n}$,
\begin{equation}
\label{eq:m4}  \Es {  e^ { - \frac{u^ \eta}{B_n} W_1 }} ^n  \leq \exp\left(  C   u ^ { \eta (\alpha+ \epsilon)}\right).
\end{equation}
Putting together \eqref{eq:m3} and \eqref{eq:m4}, we get
$$\Pr { \min_ {0 \leq i \leq  n} W_i \leq - u B_n}  \leq  \exp( -  u^ {1+ \eta}+ C  u^{ \eta ( \alpha +\epsilon)}).$$
Setting $\delta=1+\eta$, the conclusion readily follows from the choice of $\epsilon$.
\end{proof}

\subsection{Conditioned random walks and cyclic shifts}

 In view of applications for Bienaymé--Galton--Watson trees, we will need estimates on conditioned random walks. We establish in particular a conditioned version of Proposition \ref{prop:devRW}. For every $j \geq 1$,  set $\zeta_{j}= \inf \{n \geq 1: W_{n}=-j\}$.

\begin{thm} \label {thm:borneW}For every $ \delta \in (0, \alpha/ ( \alpha-1))$, there exist $C_1,C_2>0$ such that for every $ u \geq 0$ and every $n \geq 1$:
$$ \Pr { \max_ {1 \leq i \leq n} W_i \geq u B_n \ \big| \ \zeta_{1}=n} \leq C_1 \exp (- C_2 u^ \delta).$$
\end {thm}

When $\sigma^{2}<\infty$, this result is established in \cite[Eq.~(32)]{ADS13} by using a sub-exponential upper bound due to Janson \cite{Jan06b} on $\Pr{W_{n}=-m}$ valid for every $n \geq 1$ and $m \geq 0$. In the infinite variance case, a similar bound has been established in \cite[Lemma 6.6]{CK13+} when $\mu(n) \sim C \cdot n^{-(1+\alpha)}$ as $n \rightarrow \infty$, but is not known to hold in general. For this reason, we combine Proposition \ref{prop:devRW} with results of \cite{Add12} for the proof of Theorem \ref{thm:borneW}.

A useful tool for the proof of Theorem \ref{thm:borneW} is the Vervaat transform, which we now introduce. For $\bx=(x_1, \ldots,x_n)
\in \Z^n$ and $i \in \Znz$, denote by $\bx^{(i)}$ the $i$-th
cyclic shift of $\bx$ defined by $x^{(i)}_k=x_{i+k \mod n}$ for $1
\leq k \leq n$. Let $n \geq 1$ be an integer and let $\bx=(x_1,\ldots,x_{n}) \in
\Z^n$. Set $w_j=x_1+\cdots + x_j$ for $1 \leq j \leq n$ and
let the integer $i_*(\bx)$ be defined by $i_*(\bx) = \inf \{ j \geq 1; w_j= \min_{1 \leq i \leq n} w_i \}$.
The Vervaat transform of $\bx$, denoted by $\V(\bx)$, is defined to
be $\bx^{(i_*(\bx))}$. The following fact is well known (see e.g. \cite{Pit06}):

\begin {prop}\label {prop:vervaat} Under the conditional probability distribution~$\Pr{ \, \cdot \, | \, W_{n}=-1}$, the vector $\V(W_1-W_{0}, \ldots, W_ n-W_{n-1})$ has the same distribution as $(W_1-W_{0}, \ldots, W_ n-W_{n-1})$  under the conditional probability distribution~$\Pr{ \, \cdot \, | \, \zeta_{1}=n}$.
\end {prop}

\begin {proof}[Proof of Theorem \ref {thm:borneW}]To simplify notation, set $(\overline{X}_{1}, \ldots, \overline{X}_{n})=\V(W_1-W_{0}, \ldots, W_ n-W_{n-1})$. Noting that $$ \max_{1 \leq i \leq  n} \left( \overline{X}_{1}+ \cdots+\overline{X}_{i} \right) \leq  \max_ {1 \leq i \leq n} W_i  -  \min_ {1 \leq i \leq n} W_i \qquad  \textrm{on the event } W_{n}=-1,$$
Proposition \ref{prop:vervaat} gives that
\begin{eqnarray}
\label{eq:m0}
\Pr{ \max_ {1 \leq i \leq n} W_i \geq i+3  \ \big| \ \zeta_{1}=n} & \leq & \Pr{  \left. \max_ {1 \leq i \leq n} W_i  -  \min_ {1 \leq i \leq n} W_i \geq i+3  \, \right| \, W_n=-1}
\end{eqnarray}
for every $i \geq 1$ and $n \geq 1$.
Fix $m \geq i$. The proof of Eq.~(3) in \cite {Add12} shows that on the event $  \max_ {1 \leq i \leq n} W_i  -  \min_ {1 \leq i \leq n} W_i=m+3$, at least one of the following three events hold:
$$  \min_{0 \leq i \leq \lfloor n/2 \rfloor} (W_{\lfloor n/2 \rfloor}-W_{\lfloor n/2 \rfloor-i}) \leq -(m+3)/3, \qquad   \min_{0 \leq i \leq \lfloor n/2 \rfloor} (W_{n}-W_{n-i}) \leq -(m+3)/3$$
or $\min_{0 \leq i \leq \lfloor n/2 \rfloor} (W_i-W_{\lfloor n/2 \rfloor}) \leq -(m+3)/3$ (a close inspection indicates that condition d.~in \cite {Add12} should actually be $\max_{\lfloor n/2 \rfloor < i \leq n} S_{i}>(m+1)/3$ instead of  $\max_{\lfloor n/2 \rfloor < i \leq n} S_{i}>2(m+3)/3$). As a consequence, by monotonicity,
$$ \Pr {  \left.  \max_ {1 \leq i \leq n} W_i  -  \min_ {1 \leq i \leq n} W_i  \geq m+3  \, \right|  \, W_n=-1} \leq 3 \Pr { \min_ {0 \leq i \leq  \fl {n/2}} W_i \leq -(m/3+1) \, | \, W_n=-1}.$$
Then, setting $\varphi_n(j)=\Pr{W_{n}=-j}$ to simplify notation, the Markov property for the random walk $W$ applied at time $\fl{n/2}$ entails that 
$$\Pr { \min_ {0 \leq i \leq  \fl {n/2}} W_i \leq -( \frac{m}{3}+1) \, | \, W_n=-1} = \Es {1_ {  \left\{  \min_ {0 \leq i \leq  \fl {n/2}} W_i \leq -( \frac{m}{3}+1) \right\}}  \frac{ \varphi_ {n-  \fl {n/2}}( W_ { \fl {n/2} }+1 )}{ \varphi_n(1)}}.$$
But the local limit theorem \eqref{eq:LL} yields the existence of a constant $C>0$ such that 
$ { \varphi_{n-\fl{n/2}}(k)}/{\varphi_{n}(1)} \leq C$
 for every $n \geq 1$ and $k \in \Z$. Hence, by the previous estimates,
$$ \Pr { \max_ {1 \leq i \leq n} W_i \geq 3 u B_n  \ \big| \ \zeta_{1}=n } \leq  3C \Pr { \min_ {0 \leq i \leq  \fl {n/2}} W_i \leq -u B_n}  $$
for every $u \geq 0$ and $n \geq 1$. The conclusion then follows by an application of Proposition \ref {prop:devRW}.
\end {proof}

Recall from the Introduction that  $\X$ denotes the normalized excursion of a spectrally positive strictly stable Lévy process of index $\alpha$.  It is well known that the random walk $W$, conditionally on $\zeta_{1}=n$ and appropriately rescaled, converges  in distribution to $\X$ for the Skorokhod topology on $\D([0,1], \R)$ (see e.g.~\cite[Proof of Theorem 3.1]{Du03}):
$$ \left(  \frac{1}{B_{n}} W_{\fl{nt}} ; 0 \leq t \leq 1 \right)   \quad  \textrm{under} \quad  \Pr{ \, \cdot \, | \, \zeta_{1}=n}  \quad \mathop{\longrightarrow}^{(d)}_{n \rightarrow \infty} \quad ( \X_{t}; 0 \leq t \leq 1).
$$
Since the supremum is a continuous function on $\D([0,1], \R)$, we then get from Theorem \ref{thm:borneW} that for every $ \delta \in (0, \alpha/ ( \alpha-1))$ there exists $C_1,C_2>0$ such that for every $ u \geq 0$ and every $n \geq 1$:
$$ \Pr {\sup \X \geq u } \leq C_1 \exp (- C_2 u^ \delta).$$
It would be interesting to obtain an asymptotic expansion of $ \Pr {\sup \X \geq u }$ as $u \rightarrow  \infty$, similar to the one known for $\alpha=2$ (i.e. the Brownian excursion) involving a Theta function.

Finally, we will need a well-known result in the folklore of exchangeability (see \cite[Sec.~1]{Kin78}), for which we give a proof for completeness.
 Fix $n \geq 1$. A function $F: \R^{n} \rightarrow \R$ is said to be invariant under cyclic shifts if $F(\bx)=F(\bx^{(i)}) $ for every $\bx \in \mathbb{R}^{n}$ and $i \in \Z/ n \Z$. For $n \geq 1$, introduce $X_{n}=W_{n}-W_{n-1}$. Finally, recall that $\zeta_{j}= \inf \{n \geq 1: W_{n}=-j\}$. 

\begin{lem}\label{lem:shifts}Let $F: \R^{n} \rightarrow \R$ be a function invariant under cyclic shifts. Then
$$\Es{F(X_{1}, \ldots,X_{n}) \mathbbm{1}_{\zeta_{j}=n}}= \frac{j}{n} \Es{F(X_{1}, \ldots,X_{n}) \mathbbm{1}_{W_{n}=-j}}.$$
In particular,
\begin{equation}
\label{eq:eglaw}F(X_{1}, \ldots,X_{n})  \quad  \textrm{under} \quad  \Pr{ \, \cdot \, | \, \zeta_{j}=n}   \quad \mathop{=}^{(d)} \quad F(X_{1}, \ldots,X_{n})  \quad  \textrm{under} \quad  \Pr{ \, \cdot \, | \, W_{n}=-j}.
\end{equation}
\end{lem}
This result will be later used to study the maximal outdegree of (a forest of) Bienaymé--Galton--Watson trees. Its proof uses the so-called Cyclic Lemma. Before stating it, we need to introduce some notation.
For $j \geq 1$, define:
$$\S_n^{(j)}= \left\{ (x_1,\ldots,x_n) \in \{-1,0,1,2, \ldots\}^n ; \,
\sum_{i=1}^n x_i=-j \right\}$$ and
$$\Sb_n^{(j)}=  \left\{ (x_1,\ldots,x_n) \in \S_n^{(j)} ; \, \sum_{i=1}^m x_i >-j
\textrm{ for all } m \in \{0,1,\ldots,n-1\} \right\}.$$ 
For $\bx \in \S_n^{(j)}$, finally set $\I_{\textrm{\bx}}= \{ i \in \Z/n\Z; \, \bx^{(i)} \in \Sb_n^{(j)}
\}$.
The so-called Cyclic Lemma states that we have $\Card (\I_{\textnormal{\bx}} )= j$ for every $\textnormal{\bx} \in
\S_n^{(j)}$ (see \cite[Lemma 6.1]{Pit06} for a proof). 

\begin{proof}[Proof of Lemma \ref{lem:shifts}]Set $\mathbf{{X}}_{n}=( { X}_{1}, \ldots, { X}_{n})$, and note that $W _{n}=-j$ if and only if $\mathbf{{X}}^{(i)}_{n} \in \S^{(j)}_{n}$ for a (or, equivalently, every) $i \in \Z/n \Z$, that $ \zeta_{j}=n$ if and only if $\mathbf{{X}}_{n} \in \Sb^{(j)}_{n}$ and finally that $\mathbf{{X}}^{(i)}_{n} $ has the same distribution as $\mathbf{{X}}_{n} $ for every $i \in \Z/n\Z$. Then write
\begin{eqnarray*}
\Es {F\left( { X}_{1}, \ldots, { X}_{ n} \right) \mathbbm {1}_{ \{  \zeta_{j}=n\} }} &=&\Es {F \left(\mathbf {{ X}}_{n}\right)  \mathbbm {1}_{ \left\{ \mathbf{{X}}_{n} \in \Sb^{(j)}_{n} \right\} } } = \frac{1}{n} \sum_{i=0}^{n-1}  \Es {F \left(\mathbf {{ X}}^{(i)}_{n}\right)  \mathbbm {1}_{ \left\{ \mathbf{{X}}^{(i)}_{n} \in \Sb^{(j)}_{n} \right\} } } \\
&=& \frac{1}{n} \sum_{i=0}^{n-1}  \Es {F \left(\mathbf {{ X}}_{n}\right)  \mathbbm {1}_{ \left\{ \mathbf{{X}}^{(i)}_{n} \in \Sb^{(j)}_{n} \right\} } } \\
&=& \frac{1}{n}  \Es {F \left(\mathbf {{ X}}_{n}\right)  \left( \sum_{i=0}^{n-1} \mathbbm {1}_{ \left\{ \mathbf{{X}}^{(i)}_{n} \in \Sb^{(j)}_{n} \right\} } \right)    \mathbbm {1}_{ \left\{ \mathbf{{X}}_{n} \in \S^{(j)}_{n} \right\} }}  \\
&=&  \frac{1}{n}  \Es {F \left(\mathbf{{ X}}_{n}\right) \I_{{\mathbf{{X}}_{n}}}   \mathbbm {1}_{ \left\{ \mathbf{{X}}_{n} \in \S^{(j)}_{n} \right\} }}= \frac{j}{n}  \Es {F \left(\mathbf {{ X}}_{n}\right)   \mathbbm {1}_{ \left\{ \mathbf{{X}}_{n} \in \S^{(j)}_{n} \right\} }}.
\end{eqnarray*}
For the third equality, we have used the fact that $F$ is invariant under cyclic shifts, and  for the last equality we have used the Cyclic Lemma, which tells us that $\Card(\I_{{\mathbf{{X}}_{n}}}) =j$ on the event $ \mathbf{{X}}_{n} \in \S^{(j)}_{n} $. This completes the proof of the first assertion.

The second one readily follows after noting that
\begin{equation}
\label{eq:Kemp} \Pr{\zeta_{j}=n}= \frac{j}{n} \cdot \Pr{W_{n}=-j}
\end{equation}by taking $F$ to be the constant function equal to $1$.  The identity \eqref{eq:Kemp} is often referred to as Kemperman's formula in the literature. 
\end{proof}

\subsection{Bridge estimates}

We now establish a tail estimate for the supremum of a stable Lévy bridge, which will allow us to see that the exponent $\alpha/(\alpha-1)$ is optimal in Theorem \ref{thm:borneW} and which is also of independent interest. We will in addition see that this gives a tail bound for the value of the stable excursion $\X$ evaluated at a uniform point.

Denote by $(X^{\br}_{s}; 0 \leq s \leq 1)$ the stable Lévy bridge of index $\alpha$, which is roughly speaking the $\alpha$-stable Lévy process normalized such that its law at time $1$ is $Y_{\alpha}$, and conditioned to return to $0$ at time $1$ (see \cite{Cha97} or \cite[Chapter VIII]{Ber96} for a rigorous construction). Recall that $d_{\alpha}$ denotes the density of $Y_{\alpha}$.

\begin{thm}\label{thm:Xbr}For every $u > 0$, we have
$$\Pr{ \sup_{0 \leq s \leq 1}X^{\br}_{s} \geq u}= |\Gamma(-1/\alpha)| \cdot u \int_{0}^{1} ds \frac{1}{s^{1/\alpha}(1-s)^{1+1/\alpha}} d_{\alpha} \left( \frac{u}{s^{1/\alpha}} \right)d_{\alpha} \left( - \frac{u}{(1-s)^{1/\alpha}} \right). $$ 
\end{thm}

For $\alpha=2$, we will see in Remark \ref{rem:rayleigh} below that this quantity is actually equal to $e^{-u^{2}}$. However, for $\alpha \in (1,2)$ we have the following interesting asymptotic behavior:
\begin{cor}\label{cor:Xbr}For $ \alpha \in (1,2)$, we have
$$\Pr{ \sup_{0 \leq s \leq 1}X^{\br}_{s} \geq u}  \quad \mathop{\sim}_{u \rightarrow \infty} \quad  \frac{ |\Gamma(-1/\alpha)|}{\Gamma(-\alpha)} \cdot \frac{\alpha^{(4\alpha-1)/(2\alpha-2)}}{\sqrt{2\pi(\alpha-1)}}  \cdot u^{- \frac{(2+\alpha)(2\alpha-1)}{2(\alpha-1)}} \cdot e^{- (\alpha-1) \alpha^{-\alpha/(\alpha-1)} u^{\alpha/(\alpha-1)}}.$$
\end{cor}

\begin{proof}
For $\alpha \in (1,2)$, it is known that
$$ d_{\alpha}(x)  \quad \mathop{\sim}_{x \rightarrow \infty} \quad \frac{1}{\Gamma(-\alpha)} \cdot \frac{1}{x^{1+\alpha}}, \quad   d_{\alpha}(-x)  \quad \mathop{\sim}_{x \rightarrow \infty} \quad \frac{\alpha^{-1/(2\alpha-2)}}{\sqrt{2\pi(\alpha-1)}} x^{-1+ \frac{\alpha}{2(\alpha-1)}} e^{- (\alpha-1) \alpha^{-\alpha/(\alpha-1)} x^{\alpha/(\alpha-1)}}.$$
This follows from the first two terms of the asymptotic series \cite[Eq.~(2.5.4) and (2.5.17)]{Zol86}, since in the notation of the latter reference $d_{\alpha}(x)=g(x, \alpha,1)$ and $d_{\alpha}(-x)=g(x, \alpha,-1)$ (see Sec.~I.4. in \cite{Zol86} for the definition of $g$).

Since, for $s \in (0,1)$ we have $u/s^{1/\alpha}>u$ and $u/(1-s)^{1/\alpha}>u$, we can replace $d_{\alpha}(u/s^{1/\alpha})$ and $d_{\alpha}(-u/(1-s)^{1/\alpha})$ by their asymptotic equivalents and get that
\begin{eqnarray}
&& \Pr{ \sup_{0 \leq s \leq 1}X^{\br}_{s} \geq u}  \qquad \mathop{\sim}_{u \rightarrow \infty}  \qquad  | \Gamma(-1/\alpha)| \cdot  \frac{1}{\Gamma(-\alpha)} \cdot  \frac{\alpha^{-1/(2\alpha-2)}}{\sqrt{2\pi(\alpha-1)}} \cdot u^{ \frac{\alpha}{2(\alpha-1)}-1-\alpha} \notag \\
&& \qquad\qquad\qquad\qquad\qquad\qquad\qquad\qquad  \qquad  \qquad  \cdot \int_{0}^{1} ds \ \frac{s}{(1-s)^{ \frac{1}{2(\alpha-2)}+1}} e^{- c_{\alpha} \frac{u^{\alpha/(\alpha-1)}}{(1-s)^{1/(\alpha-1)}}}, \label{eq:calcul}
\end{eqnarray}
where $c_{\alpha}=(\alpha-1) \alpha^{-\alpha/(\alpha-1)}$. 
By making the change of variable $t=(1-s)^{-1/(\alpha-1)}$, we see that 
$$\int_{0}^{1} ds \ \frac{s}{(1-s)^{ \frac{1}{2(\alpha-1)}+1}} e^{- c_{\alpha} \frac{u^{\alpha/(\alpha-1)}}{(1-s)^{1/(\alpha-1)}}}=(\alpha-1) \int_{1}^{\infty} dt \ t^{-1/2} \left( 1- \frac{1}{t^{\alpha-1}} \right)  e^{-c_{\alpha} u^{\alpha/(\alpha-1)} t}.$$
It is a simple matter to check that 
$$ \int_{1}^{\infty} dt \ t^{-1/2} \left( 1- \frac{1}{t^{\alpha-1}} \right)  e^{-xt}  \quad \mathop{\sim}_{x \rightarrow \infty} \quad (\alpha-1) \frac{e^{-x}}{x^{2}}.$$
Hence, applying this with $x=c_{\alpha} u^{\alpha/(\alpha-1)}$,
$$\int_{0}^{1} ds \ \frac{s}{(1-s)^{ \frac{1}{2(\alpha-1)}+1}} e^{- c_{\alpha} \frac{u^{\alpha/(\alpha-1)}}{(1-s)^{1/(\alpha-1)}}} \quad \mathop{\sim}_{u \rightarrow \infty} \quad \alpha^{2\alpha/(\alpha-1)} \cdot u^{-2\alpha/(\alpha-1)}\cdot e^{- c_{\alpha} u^{\alpha/(\alpha-1)}}.$$
The desired estimate then follows from \eqref{eq:calcul}.\end{proof}

\begin{rem}\label{rem:rayleigh}Assume that $\alpha \in (1,2]$. For $0 \leq s \leq 1$, set
$$X^{\mathrm{br},(s)}_{t} = \begin{cases}  \Xbr_{s+t}-\Xbr_{s} & \textrm{ if } 0 \leq  t \leq 1-s \\
\Xbr_{s+t-1}-\Xbr_{s} & \textrm{ if } 1-s \leq t \leq 1.
\end{cases}$$
Set also $T^{\ast}= \inf \{t \in [0,1] : \Xbr_{t}= \inf  \Xbr \}$. It is well known that for every fixed $s \in [0,1]$, $X^{\mathrm{br},(s)}$ has the same law as $\Xbr$ and that  $X^{\mathrm{br},(T^{\ast})}$ has the same law as $\X$. In addition, if $U$ is a uniform random variable on $[0,1]$, independent of $X^{\br}$, then the random variable defined by
$$[T^{\ast}+U] \quad = \quad (T^{\ast}+U) \mathbbm{1}_{ \{T^{\ast}+U \leq 1\} } + (T^{\ast}+U-1) \mathbbm{1}_{T^{\ast}+U>1}$$
is uniform on $[0,1]$ and independent of $\Xbr$. Hence
$$ X^{\exc}_{U}   \,\, \mathop{=}^{(d)} \,\,  X^{\mathrm{br},(T^{\ast})}_{U} \,\, =  \,\, 
\Xbr_{[T^{\ast}+U]}- \inf \Xbr \,\,  \mathop{=}^{(d)} \,\,  \Xbr_{U}- \inf \Xbr \,\,  \mathop{=}  \,\,  -\inf X^{\mathrm{br},(U)} \,\, \mathop{=}^{(d)} \,\, - \inf \Xbr \,\,  \mathop{=}^{(d)} \,\, \sup \Xbr.$$
As a consequence, Theorem \ref{thm:Xbr} and Corollary \ref{cor:Xbr} hold with $\sup_{0 \leq s \leq 1}X^{\br}_{s}$ replaced by $\X_{U}$.

Also, in the case $\alpha=2$, we have $\X= \sqrt{2} \mathbbm{e}$, where $\mathbbm{e}$ is the normalized Brownian excursion. It is well known that $2 \mathbbm{e}_{U}$ is distributed as a Rayleigh random variable. Specifically, $\Pr{2 \mathbbm{e}_{U} \geq x}=e^{-x^{2}/2}$ for $x \geq 0$. In particular, for $\alpha=2$,
$$ \Pr{ \X_{U} \geq u}=\Pr{ \sup_{0 \leq s \leq 1}X^{\br}_{s} \geq u}= e^{-u^{2}}.$$
It is possible to check that Theorem \ref{thm:Xbr} gives indeed this expression for $\alpha=2$.

Finally,  since $X^{\mathrm{br},(T^{\ast})}$ has the same law as $\X$ and clearly $\sup X^{\mathrm{br},(T^{\ast})} \geq \sup \Xbr$, we have the inequality $\Pr{ \sup X^{\br}\geq u}  \leq \Pr{ \sup X^{\exc}\geq u}$ for every $u \geq 0$. Corollary \ref{cor:Xbr} thus  implies that the exponent $\alpha/(\alpha-1)$ is optimal in Theorem \ref{thm:borneW}.
\end{rem}

\begin{proof}[Proof of Theorem \ref{thm:Xbr}] Fix $u>0$. We will prove the result by a discrete approximation of the stable Lévy bridge. It is well-known that (see e.g.~\cite[Proposition 4.3]{Du03})
\begin{equation}
\label{eq:cvBR} \left(  \frac{1}{B_{n}} W_{\fl{nt}} ; 0 \leq t \leq 1 \right)   \quad  \textrm{under} \quad  \Pr{ \, \cdot \, | \, W_{n}=-1}  \quad \mathop{\longrightarrow}^{(d)}_{n \rightarrow \infty} \quad ( X^{\br}_{t}; 0 \leq t \leq 1).
\end{equation}
It is therefore enough to estimate $ \Pr{ \max_{0 \leq i \leq n} W_{i} \geq u B_{n} \ \big| \ W_{n}=-1 }$ for $u>0$ and $n \geq 1$. To this end, fix $ \eta \in (0,1)$, set $T^{(n)}= \max\{i \geq 0; W_{i}= \fl{u B_{n}} \}$. If $1 \leq j \leq n-1$, note that conditionally on the event $ \{T^{(n)}=j, W_{n}=-1\} $,  the sequence $(W_{n}-W_{n-i}; 1 \leq i \leq n-j)$ has the same distribution as $(W_i; 1 \leq i \leq n-j)$ conditionally on the event $ \{\zeta_{\fl{u B_{n}}+1}=n-j\} $.
Hence, using also \eqref{eq:Kemp}, we can write
\begin{eqnarray*}
&&  \Pr{ \max_{0 \leq i \leq n} W_{i} \geq u B_{n}, \eta n \leq T^{(n)} \leq (1-\eta) n \ \big| \ W_{n}=-1 }\\
&& \qquad\qquad\qquad  = \frac{1}{\Pr{W_{n}=-1}}\sum_{j=\eta n}^{(1-\eta)n} \Pr{W_{j}=\fl {u B_{n}}} \frac{1+\fl{u B_{n}}}{n-j} \Pr{W_{n-j}=-1-\fl{uB_{n}}}.\end{eqnarray*}
Leaving details to the reader (see e.g.~the proof of Theorem 3.1 in \cite{Kor12} for similar arguments), the dominated convergence theorem combined with the local limit theorem \eqref{eq:LL} yields
\begin{eqnarray*}
&&  \Pr{ \max_{0 \leq i \leq n} W_{i} \geq u B_{n}, \eta n \leq T^{(n)} \leq (1-\eta) n \ \big| \ W_{n}=-1 }\\
&& \qquad  \qquad   \mathop{\longrightarrow}_{n \rightarrow \infty} \quad \frac{u}{d_{\alpha}(0)} \int_{\eta}^{1-\eta} ds \frac{1}{s^{1/\alpha}(1-s)^{1+1/\alpha}} d_{\alpha} \left( \frac{u}{s^{1/\alpha}} \right)d_{\alpha} \left( - \frac{u}{(1-s)^{1/\alpha}} \right).
\end{eqnarray*}
In addition, by \eqref{eq:cvBR},
\begin{eqnarray*}
&&  \Pr{ \max_{0 \leq i \leq n} W_{i} \geq u B_{n}, T^{(n)} \not \in [\eta n, (1-\eta) n] \ \big| \ W_{n}=-1}  \\
&& \qquad\qquad\qquad\qquad  \qquad    \mathop{\longrightarrow}_{n \rightarrow \infty} \quad  \Pr{ \sup X^{\br} \geq u, \ \sup \{s \in [0,1]; X^{\br}_{s}=u \} \not \in [\eta n, (1-\eta) n]  }.
\end{eqnarray*}
By monotone convergence, the last probability tends to $0$ as $ \eta \rightarrow 1$. Therefore, for every fixed $\epsilon>0$, one may find $\eta\in (0,1)$ such that for every $n$ sufficiently large, 
$$ \left| \Pr{ \max_{0 \leq i \leq n} W_{i} \geq u B_{n} \ \big| \ W_{n}=-1 }-\Pr{ \max_{0 \leq i \leq n} W_{i} \geq u B_{n}, \eta n \leq T^{(n)} \leq (1-\eta) n \ \big| \ W_{n}=-1 } \right|  \leq \epsilon$$
and
$$ \frac{u}{d_{\alpha}(0)} \int_{[0,1] \backslash [\eta,1-\eta]} ds \frac{1}{s^{1/\alpha}(1-s)^{1+1/\alpha}} d_{\alpha} \left( \frac{u}{s^{1/\alpha}} \right)d_{\alpha} \left( - \frac{u}{(1-s)^{1/\alpha}} \right) \leq \epsilon.$$
Recalling that $d_{\alpha}(0)=|\Gamma(-1/\alpha)|^{-1}$, the conclusion readily follows.
\end{proof}

\section{Random walks and  Bienaymé--Galton--Watson trees}
\label{sec:GW}

We now explain how trees can be coded by different functions that allow to establish tail bounds for different statistics of large Bienaymé--Galton--Watson trees. In addition to Table \ref{tab:sec2}, it may be useful to refer to Table \ref{tab:sec3} to keep track of the main notation of this section.

\begin{table}[htbp]\caption{Table of the main notation and symbols appearing in Section \ref{sec:GW}.}
\centering
\begin{tabular}{c c p{10cm} }
\toprule
$\W^{\ast}(\tau)$ for $\ast \in \{\mathsf{lex},\mathsf{rev},\mathsf{bfs}\}$ & & The coding path of a tree $\tau$  obtained by using respectively the lexicographical, reverse-lexicographical and  breadth-first search ordering of the vertices of $ \tau$.\\
$W(\tau)$ && Width of a tree $\tau$. \\
$\Delta(\tau)$ && Maximum outdegree of a vertex of tree $\tau$.\\
 $\Delta^*( \X)$ && Maximum jump of $ \X$.\\
\bottomrule
\end{tabular}
\label{tab:sec3}
\end{table}

\subsection{Definitions}
  Denote by $\N = \{1, 2, \dots\}$  the set of the positive integers, set $\N^0 = \{\varnothing\}$ and let
\begin{equation*}
\mathcal{U} = \bigcup_{n \ge 0} \mathbb{N}^n.
\end{equation*}
For $u = (u_1, \dots, u_n) \in \mathcal{U}$, we denote by $|u| = n$ the length of $u$; if $n \ge 1$, we define $pr(u) = (u_1, \dots, u_{n-1})$ and for $i \ge 1$, we let $ui = (u_1, \dots, u_n, i)$; more generally, for $v = (v_1, \dots, v_m) \in \mathcal{U}$, we let $uv = (u_1, \dots, u_n, v_1, \dots, v_m) \in \mathcal{U}$ be the concatenation of $u$ and $v$. 
A plane rooted tree is a nonempty, finite subset $\tau \subset \mathcal{U}$ such that: (i) $\varnothing \in \tau$, (ii) if $u \in \tau$ with $|u| \ge 1$, then $pr(u) \in \tau$, (iii) if $u \in \tau$, then there exists an integer $k_u \ge 0$ such that $ui \in \tau$ if and only if $1 \le i \le k_u$. In the sequel, by tree we always mean plane rooted tree.

We will view each vertex $u$ of a tree $\tau$ as an individual of a population whose $\tau$ is the genealogical tree. The vertex $\varnothing$ is called the root of the tree and for every $u \in \tau$, $k_u=k_{u}(\tau)$ is the number of children (or outdegree) of $u$, $|u|$ is its generation, $pr(u)$ is its parent and more generally, the vertices $u, pr(u), pr \circ pr (u), \dots, pr^{|u|}(u) = \varnothing$ are its ancestors. If $u$ is an ancestor of $v$, we let $\llbracket u,v \rrbracket$ be the shortest path between $u$ and $v$. The size $|\tau|$ of a tree is its total number of vertices. If $\tau$ is a tree and
$u \in \tau$, we define the shift of $\tau$ at $u$ by $\theta_u \tau=\{v
\in U; \, uv \in \tau\}$, which is itself a tree.

Since $\mu$ is critical, the law of the Bienaymé--Galton--Watson tree with offspring distribution $\mu$ is the unique probability measure $\P_\mu$ on the set of all finite plane trees such that for every $j \geq 0$, $\Prmu{k_\emptyset=j}=\mu(j)$, and for every $j \geq 1$ with $\mu(j)>0$, the shifted trees
$\theta_1 \tau, \ldots, \theta_j \tau$ are independent under the conditional
probability $\Prmu{ \, \cdot \,|\, k_\emptyset=j}$ and their
conditional distribution is $\P_\mu$. A random tree whose distribution is $\P_\mu$ will be called a Bienaymé--Galton--Watson tree with offspring distribution $ \mu$, or in short a $\GW_\mu$ tree.

\subsection{Coding trees by left-continuous paths}

We associate with every ordering $u(0) \prec u(1) \prec \cdots \prec u(|\tau|-1)$ of the vertices of $ \tau$ a path $ \W(\tau)=( \W_n(\tau), 0 \leq n  \leq |\tau|)$  defined by $ \W_0(\tau)=0$ and for
$0 \leq n \leq |\tau|-1$:
$$ \W_{n+1}(\tau)= \W_{n}(\tau)+k_{u(n)}(\tau)-1.$$
Note that necessarily $ \W_{|\tau|}(\tau)=-1$.  As in \cite{ADS13}, we will use three different orderings of the vertices of a tree:
\begin{enumerate}
\item[(i)] the lexicographical ordering, where  $v \prec w$ if there exists $z \in \mathcal{U}$ such that $v = z(v_1, \dots, v_n)$, $w = z(w_1, \dots, w_m)$ and $v_1 < w_1$;
\item[(ii)] the reverse-lexicographical ordering, where  $v \prec w$ if there exists $z \in \mathcal{U}$ such that $v = z(v_1, \dots, v_n)$, $w = z(w_1, \dots, w_m)$ and $v_1 > w_1$;
\item[(iii)] the breadth-first search ordering, where $v \prec w$ if $|v| <|w|$ or  $v=z(v_1, \dots, v_m)$ and $w=z (w_1, \dots, w_m)$ with $v_{1}<w_{1}$.
\end{enumerate}
Denote by $\W^{\mathsf{lex}}(\tau),\W^{\mathsf{rev}}(\tau),\W^{\mathsf{bfs}}(\tau)$ the paths constructed by using respectively the lexicographical, reverse-lexicographical and  breadth-first search ordering of the vertices of $ \tau$. The path $\W^{\mathsf{lex}}(\tau)$ is commonly called the {\L}ukasiewicz path of $ \tau$.

We mention several useful elementary properties of these codings, which are left as an exercise to the reader.

\begin{lem}\label{lem:cod}Let $\tau$ be a tree of size $n$.
\begin{enumerate}
\item[(i)] For every $0 \leq i \leq n-1$, $\Wlex_{i}(\tau)$ is equal to the number of children of vertices of $\llbracket \emptyset, u(i) \llbracket$ branching to the right of $\llbracket \emptyset, u(i) \llbracket$.
\item[(ii)] For every $0 \leq i \leq n-1$, $\Wrev_{i}(\tau)$ is equal to the number of children of vertices of $\llbracket \emptyset, u(i) \llbracket$ branching to the left of $\llbracket \emptyset, u(i) \llbracket$.
\item[(iii)] For every $1 \leq i \leq n$,  setting $k=|u(i-1)|$, $\Wbfs_{i}(\tau)+1$ is equal to the number of children of vertices at generation $k$ less than or equal to $u(i-1)$ plus the number of vertices at generation $k$ greater than $u(i-1)$ (for the breadth-first search order).
\end{enumerate}
\end{lem}
In particular, for every $\ast \in \{\mathsf{lex},\mathsf{rev},\mathsf{bfs}\}$, $\W^{\ast}_{k}(\tau) \geq 0$ for every $0 \leq k \leq |\tau|-1$. Also, recalling that $W(\tau)$ is the width of $\tau$, an immediate consequence of (iii) is that
\begin{equation}
\label{eq:bfs}W(\tau) \leq  \max \Wbfs(\tau)+1 \leq 2 W(\tau).
\end{equation}
Indeed, of $u(i)$ is the largest (for the breadth-first ordering) vertex at generation $k$ of $\tau$, then $\Wbfs_{i+1}(\tau)+1=Z_{k+1}(\tau)$, and if $|u(i)|=k$, then $\Wbfs_{i+1}(\tau)+1 \leq Z_{k}(\tau)+Z_{k+1}(\tau)$.

It is straightforward to adapt these codings to forests. By definition, a forest $F$ is a finite ordered collection of trees $(\tau_{1}, \ldots, \tau_{k})$. One naturally extends the different orderings to $F$, by declaring that $u\prec v$ if $u \in \tau_{i}, v \in \tau_{j}$ and $1 \leq i<j \leq k$, and the coding paths $\W^{\ast}(F)$ are obtained by concatenating the jumps of $\W^{\ast}(\tau_{1}), \ldots, \W^{\ast}(\tau_{k})$.

Recall the definition of the random walk $(W_{n})_{n \geq 0}$, the notation $\zeta_{j} =\inf\{n \geq 0; \, W_n=-j\}$ for $j \geq 1$, and that $\Pmuj$ denotes the law of a forest of $j$ independent $\GW_{\mu}$ trees. The following proposition explains the importance of the different codings.

\begin{prop}\label{prop:RW}Fix $1 \leq j \leq n$. For every $\ast \in \{\mathsf{lex},\mathsf{rev},\mathsf{bfs}\}$, under $\Pmuj$, $\W^{\ast}$ has the same distribution as  $(W_0,W_1,\ldots,W_{\zeta_{j}})$. In particular, the total progeny of a $\Pmuj$ forest has the same law as $\zeta_{j}$.\end{prop}
We leave the proof to the reader (see \cite[Proposition 1.5]{LG05} for the case of the lexicographical ordering).

\subsection{Bounds on the supremum of coding paths}
 \label{sec:boundluka}
 
  Recall that $\t_{n}$ is a $\GW_{\mu}$ tree conditioned on having $n$ vertices, with $\mu$ being a critical offspring distribution belonging to the domain of attraction of a stable law of index $\alpha \in (1,2]$. 

Since  for every $\ast \in \{\mathsf{lex},\mathsf{rev},\mathsf{bfs}\}$, $\W^{\ast}( \t_{n})$  has the same distribution as $(W_{0},W_{1}, \ldots,W_{n})$ under the conditional probability distribution~$\Pr{ \, \cdot \, | \, \zeta_{1}=n}$, the following tail bound on the maximum of $\W^{\ast}(\t_{n})$ is an immediate consequence of Theorem \ref{thm:borneW}.

\begin{thm}\label{thm:coding}For every $ \delta \in (0, \alpha/ ( \alpha-1))$, there exist $C_1,C_2>0$ such that for every $\ast \in \{\mathsf{lex},\mathsf{rev},\mathsf{bfs}\}$,  for every $ u \geq 0$ and every $n \geq 1$: 
$$ \Pr { \max_ {1 \leq i \leq n} \W^{\ast}_i( \t_n) \geq u B_n} \leq C_1 \exp (- C_2 u^ \delta).$$
\end{thm}

 In particular, Theorem \ref{thm:width}  follows from this result, since $W(\t_{n}) \leq  \Wbfs(\t_{n})+1$ by \eqref{eq:bfs}. Also note that \eqref{eq:D} entails that,  for every $\ast \in \{\mathsf{lex},\mathsf{rev},\mathsf{bfs}\}$, \begin{equation}
\label{eq:D2}\left( \frac{ \W^{\ast}_ { \fl {nt}}(\t_{n})} {B_n}; 0 \leq t \leq 1\right)  \quad\mathop{\longrightarrow}^ {(d)}_{n \rightarrow \infty} \quad \X.
\end{equation}
In addition, by \eqref{eq:bfs},  $ \lim_{n \rightarrow  \infty}\Pr{W(\t_{n}) \geq u B_{n}} \geq \Pr{ \sup \X \geq 2u}$ for every $u \geq 0$,
which, combined with the last part of Remark \ref{rem:rayleigh} shows that the exponent $\alpha/(\alpha-1)$ is optimal in Theorem \ref{thm:width}.

\subsection{Bounds on the maximum outdegree of Bienaymé--Galton--Watson trees}

We are now interested in the maximum outdegree  $ \Delta( \t_n)$ of $\t_{n}$, and establish Theorems \ref{thm:degree} and \ref{thm:degree2}. The first one is a corollary of the results we have just established. Indeed, it is clear that $ \Delta(\t_{n}) \leq  \max_ {1 \leq i \leq n} \W_i( \t_n)+1$, so Theorem \ref{thm:borneW}  entails Theorem \ref{thm:degree}.

In the case $ \alpha<2$, since the maximum jump $\Delta^*( \X)$ of $ \X$ is almost surely positive, we in addition get:

\begin{prop}Assume that $ \alpha \in (1,2)$. Then, for every $p \geq 1$, $$\Es {\frac{ \Delta ( \t_n)^p}{B_n^p} }  \quad\mathop{\longrightarrow}_{n \rightarrow \infty} \quad \Es{\Delta^*( \X)^p}.$$
\end{prop}

\begin{proof}Observe that $\Delta ( \t_n)-1$ is equal to the maximum jump of $\W(\t_{n})$. Since the largest jump of a càdlàg function is a continuous functional on $\D([0,1], \R)$, the convergence \eqref{eq:D2} implies 
\begin{equation}
\label{eq:cvd}\frac{ \Delta ( \t_n)}{B_n}  \quad\mathop{\longrightarrow}^ {(d)}_{n \rightarrow \infty} \quad  \Delta^*( \X).
\end{equation} 
The claim then follows from Theorem \ref{thm:degree}.
\end{proof}

For $p=1$, it is known \cite[Prop.~3.10]{CK14} that $ \Es {\Delta^*( \X)}=  \Gamma\left(1-  \frac{1}{ \alpha} \right) \beta$, where $ \beta>0$ is the unique solution to the equation
$$ \sum_{n=0}^ \infty \frac{(-1)^n  \beta^n}{(n- \alpha) n!}=0.$$
It would be interesting to calculate the value of $ \Es{\Delta^*( \X)^p}$ for $p>1$ and also to obtain  asymptotic expansions of $ \Pr{\Delta^*( \X)>u}$ and  $ \Pr{\Delta^*( \X)<1/u}$ as $u \rightarrow  \infty$.

We now establish Theorem \ref{thm:degree2}.

\begin {proof}[Proof of Theorem \ref{thm:degree2}.] In this proof, $C>0$ will denote a constant which may change from line to line (but does not depend on $n$ or $k$). Fix $1 \leq j \leq n$.
For $\bx=(x_1, \ldots,x_n)
\in \Z^n$, set $ \max( \bx)= \max_ {1 \leq i \leq n} x_i$ and note that $\max$ is invariant under cyclic shifts. Recall that $X_{n}=W_{n}-W_{n-1}$ for $n \geq 1$. Finally, to simplify notation, set $M_i^j= \max(X_i,X_ {i+1}, \ldots, X_j)$  for every $1 \leq i \leq j \leq n$. If $F$ is a forest, observe that $\Delta(F)-1$ is equal to the maximum jump of the associated path $\W^{\ast}(F)$ with $\ast \in  \{\mathsf{lex},\mathsf{rev},\mathsf{bfs}\}$. It follows from by \eqref{eq:eglaw} and Proposition \ref {prop:RW} that
\begin{equation}\label{eq:eg}
\Esmuj{ G( \Delta(F)) \ \big| \ |F|=n}  = \Es{ G \left( M_{1}^{n} +1\right) \, \big| \, W_n=-j}
\end{equation}
for every measurable function $G: \R \rightarrow \R_{+}$.

We start by proving \eqref{eq:bd1}.  Note that $ \Pr { M_1^n \leq k-1  \, | \, W_n=-j} \leq  \Pr { M_1^ { \fl{n/2}} \leq k-1  \, | \, W_n=-j}$.
Recalling that  $\varphi_n(j)=\Pr{W_{n}=-j}$, the Markov property for the random walk $W$ applied at time $ \fl {n/2}$ yields:
$$\Prmuj{ \Delta(F) \leq  k \ | \ |F|=n} \leq  \Pr { M_1^ { \fl{n/2}} \leq k-1  \, | \, W_n=-j}= \Es{  \mathbbm {1}_ {  \{  M_1^ { \fl{n/2}} \leq k-1\}}  \frac{ \varphi_ {n-  \fl {n/2}}(W_ { \fl {n/2}}+j)}{ \varphi_n(j)}}.$$
By the local limit theorem \eqref{eq:LL}, there exists a constant $C>0$ such that
\begin{equation}
\label{eq:ll}\sup_{1 \leq j \leq M B_{n}} \frac{ \varphi_ {n-  \fl {n/2}}(W_ { \fl {n/2}}+j)}{ \varphi_n(j)} \leq C.
\end{equation}
Hence
\begin{eqnarray*}
\sup_{1 \leq j \leq M B_{n}} \Prmuj { \Delta( F) \leq  k \ | \ |F|=n}  \leq  C  \Pr { M_1^ { \fl{n/2}} \leq k-1} &=&  C  \Pr {X_1 \leq k-1}^ { \fl {n/2}} \\
&=&C_1 (1-  {\mu}([k+1,\infty)))^  { \fl {n/2}}.
\end{eqnarray*}
The inequality $ (1-x)^a \leq  \exp(-ax)$ valid for every $a \geq 0$ and $ x \in[0,1]$ yields \eqref{eq:bd1}. 

To check \eqref{eq:bd2}, write:
\begin{eqnarray*}
\Pr {M_1^n \geq  k-1  \, | \, W_n=-j} 
& \leq& \Pr { M_1^ { \fl{n/2}} \geq  k-1  \, | \, W_n=-j} +  \Pr { M_ { \fl {n/2}+1}^{n} \geq  k-1  \, | \, W_n=-j} \\
 & \leq & 2  \Pr {M_1^{ \fl{n/2}+1} \geq  k-1  \, | \, W_n=-j}
\end{eqnarray*}
As before, using the Markov property for the random walk $W$ applied at time $ \fl{n/2}+1$ combined with \eqref{eq:ll}, we get that
$$ \sup_{1 \leq j \leq M B_{n}} \Prmuj{ \Delta(F) \geq   k \ | \ |F|=n} \leq C \Pr { M_1^ { \fl{n/2}+1} \geq k-1}= C \left( 1-(1-  {\mu}([k,\infty)))^ { \fl {n/2}+1} \right).$$
The inequality $ 1-(1-x)^ a \leq   a x$ valid for every $ x \in [0,1]$ and $ a \geq 1$ shows \eqref{eq:bd2}. This completes the proof.
\end {proof}

\begin{cor} \label {cor:cor1}Let $(p_n)_ { n \geq 1}$ and $ (q_n)_ {n \geq 1}$ and be two sequence of real numbers such that $n {\mu}([p_n,\infty)) \rightarrow \infty$ and  $ n  {\mu}([ q_n,\infty)) \rightarrow 0$ as $ n \to \infty$. Then:
$$ \Pr{ p_n \leq \Delta ( \t_n) \leq   q_n}  \quad\mathop{\longrightarrow}_{n \rightarrow \infty} \quad 1$$
\end {cor}

This partially answers  \cite[Problem 19.30] {Jan12}. Note also that \eqref{eq:cvd} is stronger than Corollary \ref {cor:cor1}. Indeed, using the Potter bounds, it is possible to verify that if  $(p_n)_ {n \geq 1}$ and  $(q_n)_ {n \geq 1}$ are two sequences of positive real numbers such that $n {\mu}([p_n,\infty)) \rightarrow \infty$ and  $ n  {\mu}([ q_n,\infty)) \rightarrow 0$ as $ n \to \infty$, then $B_n/p_n \rightarrow \infty$ and $B_n/q_n \rightarrow 0$ as $n \rightarrow \infty$. By \eqref{eq:cvd}, this indeed implies that $ \Pr{ p_n \leq \Delta ( \t_n) \leq   q_n}  \rightarrow 1$.

\subsection{Application to non-crossing trees}

Let $P_{n}$ be the convex polygon inscribed in the unit disk of the complex plane whose vertices are the $n$-th roots of unity. By definition,
a \emph{non-crossing
tree} of $P_n$ is a tree drawn on the plane whose vertices are the of $P_n$ and whose edges are non-crossing line
segments. See Fig.\,\ref{fig1} for an example.

\begin{figure}[h!]
\begin {center}
\definecolor{qqttcc}{rgb}{0,0.2,0.8}
\definecolor{ccqqcc}{rgb}{0.8,0,0.8}
\definecolor{ttqqzz}{rgb}{0.2,0,0.6}
\begin{tikzpicture}[line cap=round,line join=round,>=triangle 45,x=1.0cm,y=1.0cm,scale=0.5,rotate=15]
\clip(-4.38,-3.53) rectangle (4.8,5.21);
\draw [dotted,color=ccqqcc] (4,0)-- (4,2);
\draw [dotted,color=ccqqcc] (4,2)-- (3,3.73);
\draw [dotted,color=ccqqcc] (3,3.73)-- (1.27,4.73);
\draw [dotted,color=ccqqcc] (1.27,4.73)-- (-0.73,4.73);
\draw [dotted,color=ccqqcc] (-0.73,4.73)-- (-2.46,3.73);
\draw [dotted,color=ccqqcc] (-2.46,3.73)-- (-3.46,2);
\draw [dotted,color=ccqqcc] (-3.46,2)-- (-3.46,0);
\draw [dotted,color=ccqqcc] (-3.46,0)-- (-2.46,-1.73);
\draw [dotted,color=ccqqcc] (-2.46,-1.73)-- (-0.73,-2.73);
\draw [dotted,color=ccqqcc] (-0.73,-2.73)-- (1.27,-2.73);
\draw [dotted,color=ccqqcc] (1.27,-2.73)-- (3,-1.73);
\draw [dotted,color=ccqqcc] (3,-1.73)-- (4,0);
\draw [line width=1.2pt,color=qqttcc] (4,2)-- (-2.46,-1.73);
\draw [line width=1.2pt,color=qqttcc] (-2.46,-1.73)-- (-0.73,-2.73);
\draw [line width=1.2pt,color=qqttcc] (-2.46,-1.73)-- (3,-1.73);
\draw [line width=1.2pt,color=qqttcc] (3,-1.73)-- (1.27,-2.73);
\draw [line width=1.2pt,color=qqttcc] (3,-1.73)-- (4,0);
\draw [line width=1.2pt,color=qqttcc] (-3.46,2)-- (-2.46,-1.73);
\draw [line width=1.2pt,color=qqttcc] (-3.46,2)-- (-3.46,0);
\draw [line width=1.2pt,color=qqttcc] (-3.46,2)-- (3,3.73);
\draw [line width=1.2pt,color=qqttcc] (3,3.73)-- (1.27,4.73);
\draw [line width=1.2pt,color=qqttcc] (3,3.73)-- (-0.73,4.73);
\draw [line width=1.2pt,color=qqttcc] (3,3.73)-- (-2.46,3.73);
\begin{scriptsize}
\fill [color=ttqqzz] (4,0) circle (2.5pt);
\fill [color=ttqqzz] (4,2) circle (2.5pt);
\fill [color=ttqqzz] (3,3.73) circle (2.5pt);
\fill [color=ttqqzz] (1.27,4.73) circle (2.5pt);
\fill [color=ttqqzz] (-0.73,4.73) circle (2.5pt);
\fill [color=ttqqzz] (-2.46,3.73) circle (2.5pt);
\fill [color=ttqqzz] (-3.46,2) circle (2.5pt);
\fill [color=ttqqzz] (-3.46,0) circle (2.5pt);
\fill [color=ttqqzz] (-2.46,-1.73) circle (2.5pt);
\fill [color=ttqqzz] (-0.73,-2.73) circle (2.5pt);
\fill [color=ttqqzz] (1.27,-2.73) circle (2.5pt);
\fill [color=ttqqzz] (3,-1.73) circle (2.5pt);
\end{scriptsize}
\end{tikzpicture}
 \caption{ \label{fig1} A non-crossing tree of $P_ {12}$.}
 \end{center}
\end {figure}
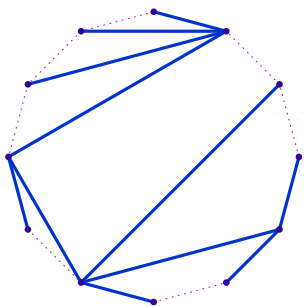
Different combinatorial properties of non-crossing trees have been studied in the literature.  Let $\mathcal {T}_n$ be  uniformly distributed over the set of all non-crossing trees of $P_n$. Deutsch \& Noy \cite {DN02} have established that, for every $c>0$, $  \Pr {   | \Delta (\mathcal {T}_n) - \log_3 n| \leq  (1+c) \log_3 \log_3 n }$ tends to $1$ as $n \rightarrow \infty$.
Using Theorem \ref {thm:degree}, we improve these bounds:

\begin{thm} \label {thm:nc}For every $c>0$, there exists constants $C_{1},C_{2}>0$ such that for every $n \geq 1$,
$$ \Pr{\Delta (\mathcal {T}_n) \geq \log_3 n + \log_3 \log_3 n +  c \log_3 \log_3 n} \leq \frac{C_{1} }{(\log_{3} n)^{c}}$$
and
$$\Pr{\Delta (\mathcal {T}_n) \leq  \log_3 n + \log_3 \log_3 n -  c \log_3 \log_3 n} \leq C_{1} e^{-C_{2} (\log_{3}n)^{c}}.$$
\end{thm}

In particular, for every $c>0$, we have $$  \Pr {|\Delta (\mathcal {T}_n) - \log_3 n - \log_3 \log_3 n| \leq  c \log_3 \log_3 n }  \quad\mathop{\longrightarrow}_{n \rightarrow \infty} \quad 1.$$

\begin{proof} In this proof, $C>0$ will denote a constant that may change from expression to expression (and which may only depend on $c$). By \cite{MP02}, $ \mathcal{T}_{n}$, viewed as a plane tree rooted at vertex $1$, is a modified Bienaymé--Galton--Watson tree $ \widetilde{\mathcal{T}}$, conditioned on having $n$ vertices, where the root has offspring distribution $\lambda(k)=2/3^{k}$ for $k \geq 1$, and all the other vertices have offspring distribution $\mu(k)=4(k+1)/3^{k+2}$ for $k \geq 0$.  Note that ${\mu}([k,\infty))=(3+2k)/3^{1+k}$ with the notation of Theorem \ref{thm:degree2}.  Also, by \cite{MP02}, $ \Pr{|\widetilde{\mathcal{T}}|=n} \sim \sqrt{3/(4\pi)} \cdot n^{-3/2}$ as $n \rightarrow \infty$. Hence, for every $k \geq 1$,
\begin{eqnarray}
 \Pr{ k_{\emptyset}( \widetilde{\mathcal{T}}) \geq k \ \big | \ |\widetilde{\mathcal{T}}|=n}&=&  \frac{1}{\Pr{ |\widetilde{\mathcal{T}}|=n}} \sum_{j \geq k}  \lambda(j) \Prmuj{|F|=n-1} \notag\\
 &=&  \sum_{j \geq k}  \lambda(j) \cdot \frac{j }{n-1} \frac{\Pr{W_{n-1}=-j}}{\Pr{ |\widetilde{\mathcal{T}}|=n}}  \leq  C \sum_{j=k}^{\infty} j \lambda(j) \leq C \frac{k}{3^{k}}. \label{eq:k}
\end{eqnarray}
For the second equality, we have used the well-known fact that $\Prmuj{|F|=n-1}=j/(n-1) \cdot \Pr{W_{n-1}=-j}$, which is for instance a consequence of Proposition \ref{prop:RW} and \eqref{eq:Kemp}, and the first inequality uses the bound $\Pr{W_{n-1}=-j} \leq C/ \sqrt{n}$ which is a consequence of the local limit theorem.
Also,  conditionally on the event $\{ k_{\emptyset}(\widetilde{\mathcal{T}}) = j, |\widetilde{\mathcal{T}}|=n\}$ the forest of trees grafted on the root of $\widetilde{\mathcal{T}}$ is a $\Pmuj$ forest conditioned on having $n-1$ vertices. Therefore, for every nonnegative measurable function $f$, we have
\begin{eqnarray*}
\Es{f(\Delta(\widetilde{\mathcal{T}})), k_{\emptyset}(\widetilde{\mathcal{T}}) \leq k \ \big| \ |\widetilde{\mathcal{T}}|=n } & \leq & \sup_{1 \leq j \leq k} \Prmuj{f(\Delta (F)) \, | \, |F|=n-1}.
\end{eqnarray*}

To prove the first inequality in Theorem \ref{thm:nc}, set $b_{n}=\log_3 n + \log_3 \log_3 n +  c \log_3 \log_3 n$ to simplify notation, and take $k= \sqrt{n} $. Then, on  one hand,  by Theorem \ref{thm:degree2} (which we may apply because $B_{n}/\sqrt{n}$ converges to a finite positive value since $\mu$ has finite variance), 
$$\sup_{1 \leq j \leq   \sqrt{n}} \Prmuj{\Delta (F) \geq b_{n} \, | \, |F|=n-1} \leq C n {\mu}( [a_{n},\infty)) \leq \frac{C}{(\log_{3} n)^{c}},$$
and, on the other hand, $\Pr{ k_{\emptyset}( \widetilde{\mathcal{T}}) \geq \sqrt{n}\ \big | \ |\widetilde{\mathcal{T}}|=n}= o \left({ (\log_{3} n)^{-c}} \right)$ by \eqref{eq:Kemp}. 

Similarly, for the second inequality, set $a_{n}=\log_3 n + \log_3 \log_3 n -  c \log_3 \log_3 n$ and also take $k= \sqrt{n} $. Then, on  one hand, by Theorem \ref{thm:degree2} 
$$\sup_{1 \leq j \leq  \sqrt{n}} \Prmuj{\Delta (F) \leq a_{n} \, | \, |F|=n-1} \leq C_{1} e^{-C_{2} n {\mu}( [a_{n}+1,\infty))} \leq  C_{1} e^{-C_{2} (\log_{3}n)^{c}},$$
and, on the other hand,  $\Pr{ k_{\emptyset}( \widetilde{\mathcal{T}}) \geq \sqrt{n}\ \big | \ |\widetilde{\mathcal{T}}|=n}= o \left( e^{-C_{2} (\log_{3}n)^{c}}\right)$ by \eqref{eq:Kemp}. 
This completes the proof.
\end{proof}

See also \cite{CKdissections,KM15} for a study of geometric properties of random non-crossing trees.

\section{Sub-exponential bounds on the height of BGW trees}
\label{sec:subexp}

Our goal is now to establish Theorem \ref{thm:height}.  Recall that $\t_{n}$ denotes a Bienaymé--Galton--Watson tree with a critical offspring distribution $\mu$ in the domain of attraction of a stable law of index $\alpha \in (1,2]$, conditioned on having $n$ vertices.

Fix $ \delta \in (1, \alpha)$. Since $ \Pr{ H( \t_n) \geq  u \cdot  {n}/{B_{n}}}=0$ for $u > B_{n}$, we shall implicitly always restrict our attention to values of $u$ such that $2 \leq u \leq B_{n}$. In addition, by monotonicity of $ x \mapsto \Prmu{H \geq x}$, we may and will always assume without loss of generality that $u n/B_{n}$ and $un /(2B_{n})$ are both integers.

 In addition to Tables \ref{tab:sec2} and \ref{tab:sec3}, it may be useful to refer to Table \ref{tab:sec4} to keep track of the main notation used and introduced in this section.

\begin{table}[htbp]\caption{Table of the main notation and symbols appearing in Section \ref{sec:subexp}.}
\centering
\begin{tabular}{c c p{12cm} }
\toprule
$H(\tau)$  & & Height of a tree $ \tau$.\\
$M(z)$ &&  Number of vertices branching off $\llbracket \emptyset, z \llbracket$ for $z \in \tau$.\\
$\mathrm{Cut}_{z}(\tau)$ && Tree obtained by keeping only $z$ and its descendance in $\tau$.\\
$Z_{k}=Z_{k}(\tau)$ && Number of vertices of height $k$ in $\tau$.\\
$\mathsf{U}_{k}=\mathsf{U}_{k}(\tau)$ && Vertex chosen uniformly at random at height $k$ in $\tau$ when $Z_{k}(\tau)>0$. \\
$\t_{n}$ && $\GW_{\mu}$ tree conditioned on having $n$ vertices.\\
$\mathcal{T}^{\ast}$ && $\GW_{\mu}$ tree conditioned to survive.\\
$\mathsf{U}_{n}^{\ast}$  &&   Vertex of the spine of $ \mathcal{T}^{\ast}$ at generation $n$.\\
$Z^{\ast}_{k}$ &&  Number of vertices of height $k$ in $ \mathcal{T}^{\ast}$.\\
$M^{\ast}_{k}$  && Number of children branching off the spine of $\mathcal{T}^{\ast}$ up to generation $k-1$.\\
$(B'_{n})_{n \geq 1}$ && An increasing scaling sequence such that $B'_{B_{n}}  \sim B_{B'_{n}}\sim n$ as $n \rightarrow \infty$.\\
\bottomrule
\end{tabular}
\label{tab:sec4}
\end{table}

\subsection{Preliminary observations}
\label{sec:obs}

We first introduce some notation. If $\tau$ is a tree and $z \in \tau$, as in \cite{ADS13}, a key quantity will be $M(z)$, the number of vertices branching off $\llbracket \emptyset, z \llbracket$, that is the number of children of vertices of $\llbracket \emptyset, z \llbracket$ that do not belong to  $\llbracket \emptyset, z \llbracket$ (the dependence in $\tau$ is implicit). Observe that  by Lemma \ref{lem:cod}, for every $z \in \tau$,
$M(z) \leq \max \W^ \textrm {lex}(\tau)+\max \W^ \textrm {rev}(\tau)$.
In addition, by Theorem \ref{thm:coding},
$$\Pr { \max \W^ \textrm {lex}( \t_n) + \max \W^ \textrm {rev}( \t_n) \geq  u^{\delta-1} B_n} \leq  2 \Pr{  \max \W^{ \textrm {lex}}( \t_n) \geq u^{\delta-1} B_{n} /2} \leq C_{1} \exp(-C_{2}u^{\delta}).$$
To establish Theorem \ref{thm:height}, it is therefore enough to show the existence of $C_{1},C_{2}>0$ such that
$$\Pr { H( \t_n) \geq  u \cdot  \frac{n}{B_{n}} \quad  \textrm{and} \quad \max_{z \in \t_{n}} M(z) \leq  u^{\delta-1} B_n } \leq C_1 \exp (- C_2 u^ \delta)$$
for every $n \geq 1$ and $2 \leq u \leq B_{n}$.

\paragraph{Size-biasing.} A key ingredient will be the notion of \emph{size-biasing}. We denote by $ \mathcal{T}^{\ast}$ a random variable having the law of the $\GW_{\mu}$ tree conditioned to survive, which we will also call the size-biased tree. This tree was introduced by Kesten \cite{Kes86}, and may informally be  described as follows. Let $X$ be a random variable with distribution $\mu$, and let $X^{\ast}$ be a random variable with the size-biased distribution $\Pr{X^{\ast}=i}=i\mu(i)$ for $i \geq 1$. In the tree $ \mathcal{T}^{\ast}$, vertices are either normal, or mutant, and the root is a mutant vertex. Normal nodes have outdegree distributed according to independent copies of $X$, while mutant nodes have outdegree distributed according to independent copies of $X^{\ast}$. All children of a normal vertex are normal, while for each mutant node, all of its children are normal, except one, selected uniformly at random, which is mutant. In particular, the tree $ \mathcal{T}^{\ast}$ has a unique infinite simple path starting from the root, called the spine. Vertices on the spine have outdegree distributed according to $X^{\ast}$, while all the other vertices  have outdegree distributed according to $\mu$. For every $n \geq 0$, we let $\mathsf{U}_{n}^{\ast}$ be the vertex of the spine of $ \mathcal{T}^{\ast}$ at generation $n$.

We now introduce some notation. Let $\tau$ be a tree and $z \in \tau$. Recall that $\theta_z \tau=\{w
\in U; \, zw \in \tau\}$. Set $\mathrm{Cut}_{z}(\tau)= \tau \backslash \{zw ; w \in \theta_{z} \tau \backslash \{\emptyset\} \}$, which is a tree such that $z \in \mathrm{Cut}_{z}(\tau)$. If $n \geq 0$ is such that $Z_{n}(\tau)>0$, we let $\mathsf{U}_{n}$ be a vertex chosen uniformly at random among all those at generation $n$ of $ \tau$. The term \emph{size-biasing} comes from the following result, see \cite{LPP95b}: for any nonnegative functions $G_{1},G_{2}$:
\begin{equation}
\label{eq:sizeb} \Esmu{\mathbbm{1}_{Z_{n}>0} F(Z_{n}) G_{1}(\mathrm{Cut}_{\mathsf{U}_{n}}( \tau),\mathsf{U}_{n}) G_{2}( \theta_{\mathsf{U}_{n}} \tau)  }= \Es{ \frac{F(Z_{n}( \Tca ))}{Z_{n}( \Tca) }  G_{1}(\mathrm{Cut}_{\mathsf{U}_{n}^{\ast}}( \Tca),\mathsf{U}_{n}^{\ast})} \cdot \Esmu{G_{2}( T)}.
\end{equation}
This identity is usually written and used with $F$ replaced by $ x \mapsto x F(x)$, but we have written it in this form in view of future use. 

\paragraph{Two technical estimates.} We first establish an estimate describing the asymptotic behavior of $B_{n}$.

\begin{lem}\label{lem:pB}We have $B_{n} \cdot  \Prmu{Z_{n/B_{n}}>0} \rightarrow (\alpha-1)^{-1/(\alpha-1)} $ as $n \rightarrow  \infty$.
\end{lem}

\begin{proof}By \eqref{eq:equiv0} and \cite[Lemma 2]{Sla68},
\begin{equation}
\label{eq:equivp} \Prmu{Z_{n}>0}^{\alpha-1} L \left( {\Prmu{Z_{n}>0}} ^{-1}\right)   \quad\mathop{\sim}_{n \rightarrow \infty} \quad  \frac{\alpha}{\Gamma(3-\alpha) n}.
\end{equation}
To simplify notation, in this proof we set $R_{n}={\Prmu{Z_{n/B_{n}}>0}} ^{-1}$, so that $R_{n} \rightarrow \infty$, and we aim at showing that $B_{n}/R_{n} \rightarrow (\alpha-1)^{-1/(\alpha-1)}$ as $n \rightarrow \infty$. By replacing $n$ with $n/B_{n}$ in \eqref{eq:equivp}  and using \eqref{eq:cvB}, we get that
\begin{equation}
\label{eq:RB}\left(  \frac{B_{n}}{R_{n}} \right) ^{\alpha-1}  \quad \mathop{\sim}_{n \rightarrow \infty} \quad (\alpha-1)^{-1}\cdot \frac{L(B_{n})}{L(R_{n})}.
\end{equation}
Since $R_{n}$ and $B_{n}$ tend to $\infty$ as $n \rightarrow \infty$, this implies that $(B_{n}/R_{n})^{\alpha-1} \rightarrow  (\alpha-1)^{-1}$. Indeed, if $(B_{n}/R_{n})^{\alpha-1} $ converges to a positive limit along a subsequence, then $L(B_{n})/L(R_{n}) \rightarrow 1$ along this subsequence by properties of slowly varying functions. If $B_{n}/R_{n} \rightarrow  \infty$ along a subsequence, then it is a simple matter to see that a contradiction arises from \eqref{eq:RB} by the Potter bounds. One similarly treats the case $B_{n}/R_{n} \rightarrow  0$ along a subsequence, and this completes the proof.
\end{proof}

 Let $(S_{i})_{i \geq 1}$ be a sequence of i.i.d.~random variables having the distribution of the total size of a $\GW_{\mu}$ tree. Recall that $d_{\alpha}$ is the density of a random variable $Y_{\alpha}$ with Laplace exponent given by $\Es{ \exp(- \lambda Y_{\alpha})}= \exp(\lambda ^ \alpha)$ for every $\lambda>0$. It is a standard fact that 
\begin{equation}
\label{eq:Tn} \Prmu{|\tau|=n }  \quad \mathop{\sim}_{n \rightarrow \infty} \quad  \frac{d_{\alpha}(0)}{n B_{n}}.
\end{equation}
This can for instance be seen by combining the equality $ \Prmu{|\tau|=n } =\Pr{W_{n}=-1}/n$ with the local limit theorem. In particular, $ \Prmu{|\tau| \geq n }  \sim  \alpha d_{\alpha}(0) /B_{n} $ as $n \rightarrow \infty$. Hence under $\Pmu$, $|\tau|$ belongs to the domain of attraction of a stable law of index $1/\alpha$.
Therefore, setting
\begin{equation}
\label{eq:Bprime}B'_n =  \inf \left\{ x \geq 0; \,  \Prmu{|\tau| \geq x} \leq
\frac{\alpha d_{\alpha}(0)}{n}\right\},
\end{equation}
the proof of Theorem 2.6.1 in \cite{IL71} (in particular, our $B'_{n}$ plays the role of $B_{n}$ that is defined in \cite[Eq.~(2.6.6)]{IL71})  shows that  $(S_{1}+ \cdots+S_{n})/B'_{n}$ converges in distribution to a stable law of index $1/\alpha$ as $n \rightarrow \infty$ (since $1/\alpha<1$, by \cite[Theorem 3 (i) in Chap.~XVII.5]{Fel71} no centering procedure is required). Hence, if we set $\Psi_{n}(k)= \Pr{S_{1}+ \cdots+S_{n}=k}$, by the local limit theorem  \cite[Theorem 4.2.1]{IL71},
$$ \lim_{n \rightarrow \infty} \sup_{k \in \Z} \left| B'_{n} \Psi_{n}(k)-g \left(  \frac{k}{B'_{n}} \right) \right| =0,$$
where $g$ is the density of a stable random variable with index $1/\alpha$.  By \cite[Eq.~(I.20)]{Zol86}), $g$ is bounded. As a consequence, there exists a 
constant $C>0$ such that
\begin{equation}
\label{eq:psi} \textrm{for every } n \geq 1, \quad  \textrm{ for every } k \geq 1,  \qquad \Psi_{n}(k) \leq  \frac{C}{B'_{n}}.
\end{equation}

We will also later use the following asymptotic estimate involving both $B_{n}$ and $B'_{n}$:
\begin{equation}
\label{eq:BB} B'_{B_{n}}  \quad \mathop{\sim}_{n \rightarrow \infty} \quad n.
\end{equation}
To see this, note that, by \eqref{eq:Bprime}, we have  $ \Prmu{|\tau| \geq B'_{n} }  \sim  \alpha d_{\alpha}(0) / n$ as $n \rightarrow \infty$. Since $ \Prmu{|\tau| \geq  m  }  \sim  \alpha d_{\alpha}(0) /B_{m} $ as $m \rightarrow \infty$, we also get that $ \Prmu{|\tau| \geq B'_{n} }  \sim  \alpha d_{\alpha}(0) / B_{B'_{n}}$ as $n \rightarrow \infty$. This implies that $B_{B'_{n}} \sim n$ as $n \rightarrow \infty$. The estimate \eqref{eq:BB} then immediately follows from \cite[p.~21]{Sen76}.

\subsection{Proof of Theorem \ref{thm:height}}

We start by explaining the main steps for proving Theorem \ref{thm:height}. Recall that we assume that $u n/B_{n}$ and $un /(2B_{n})$ are both integers. First, let $Z^{\ast}_{k}=Z_{k}( \mathcal{T}^{\ast})$ be the size of the $k$-th generation of the size biased tree, let also $M^{\ast}_{k}=M(\mathsf{U}^{\ast}_{k})$ be the number of children branching off the spine of $\mathcal{T}^{\ast}$ up to generation $k-1$. By using a cutting and size-biasing argument, we shall prove the following bound.

\begin{lem}\label{lem:sb} For every $n \geq 1$ and $2 \leq u \leq B_{n}$, we have
\begin{eqnarray}
&&  \Pr { H( \t_n) \geq  u \cdot  \frac{n}{B_{n}} \quad  \textrm{and} \quad \max_{z \in \t_{n}} M(z) \leq u^{\delta-1} B_n } \notag\\
&& \qquad  \qquad  \leq   \frac{1}{ \Prmu{ |\tau|=n }}  \cdot \Es{ \frac{1}{Z^{\ast}_{un/(2B_{n})}}} \cdot \Es{ \frac{1}{B'_{M^{\ast}_{un/(2B_{n})}}}   \mathbbm{1}_{M^{\ast}_{un/(2B_{n})} \leq u^{\delta-1} B_{n}}   }. \label{eq:infame}
\end{eqnarray}
\end{lem}

The most technical part of the proof consists in bounding the second expectation appearing in \eqref{eq:infame}:

\begin{lem}\label{lem:exp} There exist $C_{1},C_{2}>0$ such that
$$n \cdot \Es{ \frac{1}{B'_{M^{\ast}_{un/(2B_{n})}}}   \mathbbm{1}_{M^{\ast}_{un/(2B_{n})} \leq u^{\delta-1} B_{n}}   }  \leq C_1 \exp (- C_2 u^ \delta). $$
for every $n \geq 1$ and $2 \leq u \leq B_{n}$.
\end{lem} 

We now explain how Theorem \ref{thm:height} follows from these two results.

\begin{proof}[Proof of Theorem \ref{thm:height}] Recall that it is enough to  show the existence of $C_{1},C_{2}>0$ such that
\begin{equation}
\label{eq:amontrer2} \Pr { H( \t_n) \geq  u \cdot  \frac{n}{B_{n}} \quad  \textrm{and} \quad \max_{z \in \t_{n}} M(z) \leq  u^{\delta-1} B_n } \leq C_1 \exp (- C_2 u^ \delta).
\end{equation}
for every $n \geq 1$ and $2 \leq u \leq B_{n}$. By \eqref{eq:sizeb},  $\Es{1/Z^{\ast}_{un/(2B_{n})}}=  \Prmu{Z_{un/(2B_{n})}>0} \leq  \Prmu{Z_{n/B_{n}}>0}$ since $u \geq 2$.  Hence, by Lemmas \ref{lem:sb} and \ref{lem:exp},
$$\Pr { H( \t_n) \geq  u \cdot  \frac{n}{B_{n}} \quad  \textrm{and} \quad \max_{z \in \t_{n}} M(z) \leq u^{\delta-1} B_n } \leq \frac{ \Prmu{Z_{n/B_{n}}>0}}{n \cdot \Prmu{ |\tau|=n }} \cdot C_1 \exp (- C_2 u^ \delta).$$
But by \eqref{eq:Tn} and Lemma \ref{lem:pB}, the quantity $ \Prmu{Z_{n/B_{n}}>0}/(n \cdot \Prmu{ |\tau|=n }) $ converges to a positive limit as $ n \rightarrow \infty$. The conclusion follows.
\end{proof}

We now prove Lemma \ref{lem:sb}.

\begin{proof}[Proof of Lemma \ref{lem:sb}]By using the size-biasing relation \eqref{eq:sizeb}, write 
\begin{eqnarray*}
&&  \Pr { H( \t_n) \geq  u \cdot  \frac{n}{B_{n}} \quad  \textrm{and} \quad \max_{z \in \t_{n}} M(z) \leq u^{\delta-1} B_n } \\
&& \qquad  \qquad\qquad  \leq \quad   \Pr { H( \t_n) \geq  u \cdot  \frac{n}{B_{n}} \quad  \textrm{and} \quad M(\mathsf{U}_{ un/B_{n}}) \leq u^{\delta-1} B_n } \\
&& \qquad  \qquad \qquad   = \quad \frac{1}{ \Prmu{ |\tau|=n }} \Es{ \frac{1}{Z^{\ast}_{un/B_{n}}}    \mathbbm{1}_{M^{\ast}_{un/B_{n}} \leq u^{\delta-1} B_{n}}  \mathbbm{1}_{ |\mathrm{Cut}_{\mathsf{U}^{\ast}_{un/B_{n}}}( \mathcal{T}^{\ast}) |+| \mathcal{T}|=n+1}},
\end{eqnarray*}
where $ \mathcal{T}$ is an independent $ \GW_{\mu}$ tree. Observe that the random variables $Z^{\ast}_{un/B_{n}}$ and $M^{\ast}_{un/B_{n}}$ are of course not independent, which is a major issue. Also, forgetting the other terms, note that one should  find a good number of independent trees in order to bound $\Pr{ |\mathrm{Cut}_{\mathsf{U}^{\ast}_{un/B_{n}}}( \mathcal{T}^{\ast})|+| \mathcal{T}|=n+1}$ by using \eqref{eq:psi}. 

The main idea is to introduce independence: roughly speaking, we cut the spine up to generation $un/B_{n}$ in half and, denoting by respectively $ \mathcal{S}_{\downarrow}$ and $ \mathcal{S}^{\uparrow}$ its lower and upper part, we bound from below $Z^{\ast}_{un/B_{n}}$ by the number of vertices at generation $un/B_{n}$ in $ \mathcal{T}^{\ast}$ that have an ancestor belonging to $\mathcal{S}^{\uparrow}$,  we bound from below $M^{\ast}_{un/B_{n}}$ by the number of children branching off $ \mathcal{S}_{\downarrow}$, and the collection of independent trees we use for applying the local limit theorem are those branching off $ \mathcal{S}_{\downarrow}$.

Specifically, define  $ \mathcal{T}^{\ast}_{\downarrow}=\mathrm{Cut}_{\mathsf{U}_{un/(2B_n)}^{\ast}}( \Tca)$ and  $ \mathcal{T}^{\ast}_{\uparrow}= \theta_{\mathsf{U}_{un/(2B_n)}^{\ast}}\mathrm{Cut}_{\mathsf{U}_{un/B_n}^{\ast}}( \Tca) $ (to simplify notation, we keep the dependence in $u$ and $n$ implicit).  Note that $ \mathcal{T}^{\ast}_{\downarrow}$ and $ \mathcal{T}^{\ast}_{\uparrow}$ are independent and have same distribution, and also that $M^{\ast}_{un/(2B_{n})}$ is a measurable function of $ \mathcal{T}^{\ast}_{\downarrow}$. Finally, note that 
$$ \left |\mathrm{Cut}_{\mathsf{U}^{\ast}_{un/B_{n}}}( \mathcal{T}^{\ast}) \right|=|\mathcal{T}^{\ast}_{\downarrow}| +| \mathcal{T}^{\ast}_{\uparrow}|-1$$
and that $|\mathcal{T}^{\ast}_{\downarrow}|$ is equal to $1+un/(2B_{n})$ plus the total size of a forest of $M^{\ast}_{un/(2B_{n})}$ independent $\GW_{\mu}$ trees. Observing that $Z^{\ast}_{un/B_{n}} \geq Z_{un/(2B_{n})}(\mathcal{T}^{\ast}_{\uparrow})$ and $M^{\ast}_{un/B_{n}} \geq M^{\ast}_{un/(2B_{n})}$, we therefore have
\begin{eqnarray*}
&& \Es{ \frac{1}{Z^{\ast}_{un/B_{n}}}    \mathbbm{1}_{M^{\ast}_{un/B_{n}} \leq u^{\delta-1} B_{n}}  \mathbbm{1}_{ |\mathrm{Cut}_{\mathsf{U}^{\ast}_{un/B_{n}}}( \mathcal{T}^{\ast})|+| \mathcal{T}|=n+1}}\\
&& \qquad  \qquad     \leq \Es{ \frac{1}{Z_{un/(2B_{n})}(\mathcal{T}^{\ast}_{\uparrow})}    \mathbbm{1}_{M^{\ast}_{un/(2B_{n})} \leq u^{\delta-1} B_{n}}   \Psi_{M^{\ast}_{un/(2B_{n})}
}(n+1-|\mathcal{T}^{\ast}_{\uparrow}|- | \mathcal{T}|-{un/(2B_{n})}) }, 
\end{eqnarray*}
where we recall that  $\Psi_{n}(k)$ be the probability that a forest of $n$ independent $ \GW_{\mu}$ trees has total size $k$. The desired result then follows by using \eqref{eq:psi} and noting that   $Z_{un/(2B_{n})}(\mathcal{T}^{\ast}_{\uparrow})$ is independent of $M^{\ast}_{un/(2B_{n})}$ and has the same distribution as $Z^{\ast}_{un/(2B_{n})}$.\end{proof}

At this point, we make several comments concerning Lemma \ref{lem:sb}. First, in its proof, we chose to use \eqref{eq:psi} with the forest of trees branching off the lower half of the spine. It would perhaps been more natural to choose the forest of trees above generation $un/(2B_{n})$ in the tree $\mathcal{T}^{\ast}_{\uparrow}$, and in this way get that for every $n \geq 1$ and $2 \leq u \leq B_{n}$,
\begin{eqnarray*}
&&  \Pr { H( \t_n) \geq  u \cdot  \frac{n}{B_{n}} \quad  \textrm{and} \quad \max_{z \in \t_{n}} M(z) \leq u^{\delta-1} B_n } \notag\\
&& \qquad  \qquad  \leq   \frac{1}{ \Prmu{ |\tau|=n }}  \cdot \Es{ \frac{1}{Z^{\ast}_{un/(2B_{n})} B'_{Z^{\ast}_{un/(2B_{n})}}} } \cdot \Pr{M^{\ast}_{un/(2B_{n})} \leq u^{\delta-1} B_{n}}.
\end{eqnarray*}
Unfortunately, this does not allow to conclude, since
$$\Es{ \frac{n B_{n}}{Z^{\ast}_{un/(2B_{n})} B'_{Z^{\ast}_{un/(2B_{n})}}} }  \quad \mathop{\longrightarrow}_{n \rightarrow \infty} \quad \infty.$$
The reason is that $Z^{\ast}_{n/B_{n}}/B_{n}$ converges in distribution to a positive random variable whose density has polynomial decay near $0$ (see the discussion after Proposition \ref{prop:Z}), while $M^{\ast}_{n/B_{n}}/B_{n}$ converges in distribution to a positive random variable whose density has exponential decay near $0$.

Also, it would have been slightly simpler to bound from below $M^{\ast}_{un/(2B_{n})}+1$ by the maximal outdegree of a vertex of the lower half of the spine. One may indeed follow this path when $ \alpha \in (1,2)$. However, in the case $\alpha=2$, these two quantities are not of the same order.

It thus remains to establish Lemma \ref{lem:exp}. The main technical estimate is the following.

\begin{lem}\label{lem:inflame}Fix $c_{0},C_{0}>0$ and $\alpha_{1}>\alpha$. There exist $\gamma>0$ and constants $C_{1},C_{2}>0$ such that
$$\Pr{{M^\ast_{un/(2B_{n})}} \leq C_{0} \frac{{u^{\delta-1}B_{n}}}{t^{1/\alpha_{1}}} } \leq C_{1}e^{- C_{2}u^{\delta} t^{\gamma}}$$
for every $n \geq 1$, $2 \leq u \leq B_{n}$ such that $ c_{0} u n/B_{n} \leq u^{\delta-1} B_{n}$ and $1 \leq t \leq {B'_{u^{\delta-1}B_{n}}}/{B'_{c_{0}un/B_{n}}}$.
\end{lem}

We postpone its proof and explain how Lemma \ref{lem:exp} follows from Lemma \ref{lem:inflame}.

\begin{proof}[Proof of Lemma  \ref{lem:exp}] We start with a preliminary result. We claim that there exists $c_{0},C_{1},C_{2}>0$ such that
$$ n \Pr{M^{\ast}_{un/(2B_{n})} \leq  c_{0} un/B_{n}} \leq  C_{1}e^{-C_{2}u^{\delta}}$$
for every $n \geq 1$ and $2 \leq  u \leq  B_{n}$. To show this, we will use the following simple fact, which follows from Markov's exponential inequality. Let $S_{n}$ be a sum of $n$ i.i.d~Bernoulli random variables of parameter $p \in (0,1)$. Then, for every $p_{0} \in (0,p/2)$, there exists a constant $C>0$ such that $\Pr{S_{n} \leq 2 p_{0} n} \leq \exp(-C n)$
for every $n \geq 1$. 

Now write $M^{\ast}_{un/(2B_{n})}= \sum_{i=1}^{un/(2B_{n})} (X_{i}^{\ast}-1)$, where $(X_{i}^{\ast})_{i \geq 1}$ are i.i.d.~random variables having the same distribution as the size-biased random variable $X^{\ast}$, so that $S^{\ast}_{n}= \# \{1 \leq i \leq n ; {X}^{*}_{i}>1\}$, which is a sum of $n$ i.i.d.~Bernoulli random variables of parameter $1-\mu(1)$, and note that $M^{\ast}_{un/(2B_{n})} \geq S^{\ast}_{un/(2B_{n})}$. Therefore, choosing $ c_{0} \in (0,(1-\mu(1))/2)$, by the previous discussion we get
$$ \Pr{M^{\ast}_{un/(2B_{n})} \leq  c_{0} un/B_{n}} \leq \Pr{S^{\ast}_{un/(2B_{n})} \leq  c_{0} un/B_{n}} \leq \exp(-C u n/B_{n})$$
Then observe that 
$$n  \exp(-C u n/B_{n}) \leq C_{1}\exp(- u^{\delta}),$$
for a certain $C_{1}>0$, for every $n \geq 1$ and $2 \leq  u \leq  B_{n}$.  Indeed, set $f_{n}(u)= C un/B_{n}-u^{\delta}- \ln(n)$. Since $\delta \in (1,\alpha)$, $f_{n}$ is concave so that $ \inf_{[2,B_{n}]} f_{n} \geq \min(f_{n}(2),f_{n}(B_{n}))$. An application of the Potter bounds show that $f_{n}(2)$ and $f_{n}(B_{n})$ both tend to infinity as $n \rightarrow  \infty$.  This completes the proof of the claim.

As a consequence, Lemma \ref{lem:exp} will follow if we manage to show  the existence of $C_{1},C_{2}>0$ such that
$$\Es{ \frac{n}{B'_{M^{\ast}_{un/(2B_{n})}}}   \mathbbm{1}_{ c_{0} u n/B_{n} \leq  M^{\ast}_{un/(2B_{n})} \leq u^{\delta-1} B_{n}}   }  \leq C_1 \exp (- C_2 u^ \delta). $$
for every $n \geq 1$ and $2 \leq u \leq B_{n}$ such that $ c_{0}u n/B_{n} \leq u^{\delta-1} B_{n}$. To this end, using the formula
$$ \Es{ \frac{1}{X} \mathbbm{1}_{x \leq X \leq y}}= \frac{1}{y} \int_{1}^{y/x} dt \ \Pr{x \leq X \leq y/t}+ \frac{1}{y} \cdot \Pr{x \leq X \leq y}$$
valid for every $0 \leq x \leq y$ and every nonnegative real valued random variable $X$,
 write
\begin{eqnarray*}
&&  \Es{ \frac{n}{B'_{M^{\ast}_{un/(2B_{n})}}}   \mathbbm{1}_{ c_{0} u n/B_{n} \leq  M^{\ast}_{un/(2B_{n})} \leq u^{\delta-1} B_{n}}   } \\
&&   \quad   \leq  \frac{n}{B'_{u^{\delta-1}B_{n}}} \cdot \int_{1}^{
\frac{B'_{u^{\delta-1}B_{n}}}{B'_{c_{0}un/B_{n}}}} dt \ \Pr{B'_{M^\ast_{un/(2B_{n})}} \leq  \frac{B'_{u^{\delta-1}B_{n}}}{t} }+  \frac{n}{B'_{u^{\delta-1}B_{n}}} \Pr{M^{\ast}_{un/(2B_{n})} \leq u^{\delta-1} B_{n}  }  \\
\end{eqnarray*}
If $ \alpha_{1}>\alpha$ is fixed, by the Potter bounds there exists a constant $C>0$ such that  $ t^{-1} \cdot {B'_{u^{\delta-1}B_{n}}} \leq B'_{ C {u^{\delta-1}B_{n}} \cdot{t^{-1/\alpha_{1}}}}$ for every $t \geq 1$, $n \geq 1$ and $2 \leq u \leq B_{n}$ such that $ c_{0}u n/B_{n} \leq u^{\delta-1} B_{n}$. Therefore, using the fact that $B'$ is increasing, we get that 
$$\Pr{B'_{M^\ast_{un/(2B_{n})}} \leq  t^{-1} \cdot {B'_{u^{\delta-1}B_{n}}} } \leq \Pr{{M^\ast_{un/(2B_{n})}} \leq C \frac{{u^{\delta-1}B_{n}}}{t^{1/\alpha_{1}}} }.$$
In addition, by  \eqref{eq:BB} and the Potter bounds, there exist $C,\eta>0$ such that ${n} \cdot ({B'_{u^{\delta-1}B_{n}}})^{-1} \leq C u^{\eta}$ for every $n \geq 1$ and $u \geq 2$. Therefore, by Lemma \ref{lem:inflame},
$$\Es{ \frac{n}{B'_{M^{\ast}_{un/(2B_{n})}}}   \mathbbm{1}_{ c_{0} u n/B_{n} \leq  M^{\ast}_{un/(2B_{n})} \leq u^{\delta-1} B_{n}}  } \leq   C u^{\eta}  \left( \int_{1}^{\infty} dt \, e^{- C_{2}u^{\delta} t^{\gamma}}+e^{-C_{2}u^{\delta}} \right) .$$
Finally note that, for every $A>0$,
$$ \int_{1}^{\infty} dt \, e^{-A t^{\gamma}} \leq \int_{1}^{\infty} dt \, \frac{A \gamma t^{\gamma-1}}{\gamma A} e^{-A t^{\gamma}}= \frac{1}{\gamma A} \left[- e^{-A t^{\gamma}} \right]_{1}^{\infty}= \frac{e^{-A}}{\gamma A}.$$
Therefore
$$\Es{ \frac{n}{B'_{M^{\ast}_{un/(2B_{n})}}}   \mathbbm{1}_{ c_{0} u n/B_{n} \leq  M^{\ast}_{un/(2B_{n})} \leq u^{\delta-1} B_{n}}  }  \leq C u^{\eta-\delta} e^{-C_{2} u^{\delta}},$$
and the proof is complete.
\end{proof}

It remains to establish Lemma \ref{lem:inflame}.

\begin{proof}[Proof of Lemma \ref{lem:inflame}] In this proof, $C$ denotes a positive constant that may change from line to line (and that does not depend on $u$ and $n$). We mention that the condition $ c_{0} u n/B_{n} \leq u^{\delta-1} B_{n}$ restricts $u$ to take smaller values than possible and will play a crucial role.

 First, fix $ \alpha_{+} \in (\alpha, \alpha_{1})$ so that  by the Potter bounds, if $a_{n} \leq b_{n}$ are positive sequences tending to infinity, $B'_{b_{n}}/B'_{a_{n}} \leq  C (b_{n}/a_{n})^{\alpha_{+}}$ for every $n \geq 1$. Then let $\gamma>0$ be such that
$$0< \gamma < \min \left( \frac{1}{\alpha_{+}(2-\delta)}, \frac{1}{\alpha_{1}(2-\alpha)} \right) .$$
Since we may write $M^{\ast}_{un/(2B_{n})}= \sum_{i=1}^{un/(2B_{n})} (X_{i}^{\ast}-1)$, where $(X_{i}^{\ast})_{i \geq 1}$ are i.i.d.~random variables having the same distribution as the size-biased random variable $X^{\ast}$, by Markov's exponential inequality, for $\lambda>0$,
$$\Pr{{M^\ast_{un/(2B_{n})}} \leq C_{0} \frac{{u^{\delta-1}B_{n}}}{t^{1/\alpha_{1}}} }  \leq \exp \left( \lambda C_{0} \frac{{u^{\delta-1}B_{n}}}{t^{1/\alpha_{1}}}  +  \frac{un}{2 B_{n}} \ln \Es{e^{-\lambda (X^{*}-1) }}\right).$$
We take $ \lambda= u t^{ \gamma}/B_{n}$ (the dependence in $n$ and $u$ of $ \lambda$ is implicit).  

We  check that $ \lambda \rightarrow  0$. First, since $c_{0} un/B_{n} \leq u^{\delta-1} B_{n}$ and $t \leq {B'_{u^{\delta-1}B_{n}}} \cdot ( {B'_{c_{0}un/B_{n}}})^{-1}$, we have
\begin{equation}
\label{eq:bornes}2 \leq u \leq C \frac{B_{n}^{2/(2-\delta)}}{n^{1/(2-\delta)}}, \qquad  1 \leq t \leq \frac{B_{n}^{2\alpha_{+}}}{u^{\alpha_{+}(2-\delta)} n^{\alpha_{+}} }.
\end{equation}
By combining the previous two estimates, we get that
$$ \frac{u t^{\gamma}}{B_{n}} \leq u^{1-\alpha_{+}(2-\delta) \gamma} \cdot \frac{B_{n}^{2\alpha_{+}\gamma-1}}{n^{\alpha_{+}\gamma}} \leq  \left( C \frac{B_{n}^{2/(2-\delta)}}{n^{1/(2-\delta)}} \right)^{1-\alpha_{+}(2-\delta) \gamma} \cdot \frac{B_{n}^{2\alpha_{+}\gamma-1}}{n^{\alpha_{+}\gamma}}= C  \left( \frac{B_{n}^{\delta}}{n} \right) ^{ \frac{1}{2-\delta}},$$
where we have used the fact that $1-\alpha_{+}(2-\delta) \gamma>0$ for the second inequality. Since $\delta<\alpha$, $B_{n}^{\delta}/n \rightarrow 0$ as $ n \rightarrow  \infty$ by the Potter bounds. This shows that $ \lambda \rightarrow 0$. Note that the convergence $\lambda \rightarrow 0$ does not hold without the restriction $ c_{0} u n/B_{n} \leq u^{\delta-1} B_{n}$ (take e.g.~$u=B_{n}$ and $t=1$).

Now let $ \epsilon>0$ be such that $ \delta+ \epsilon< \alpha$ and $\gamma(\alpha-1)-\epsilon\gamma>\gamma-1/\alpha_{1}$. By the Potter bounds, $L(B_{n}/(ut^{\gamma})) \geq C L(B_{n}) (ut^{\gamma})^{-\epsilon}$. Thus, using \eqref{eq:equiv3} and \eqref{eq:cvB},
\begin{eqnarray*}
 \frac{un}{2 B_{n}} \ln \Es{e^{-\lambda (X^{*}-1)}} \leq -C u^{\alpha} t^{\gamma(\alpha-1)} \frac{n}{B_{n}^{\alpha}} \cdot L \left(  \frac{B_{n}}{u t^{\gamma}} \right)  &=&  -C u^{\alpha} t^{\gamma(\alpha-1)} \frac{n L(B_{n})}{B_{n}^{\alpha}} \cdot \frac{L \left(  \frac{B_{n}}{u t^{\gamma}} \right)}{L(B_{n})}\\
 & \leq &  -C u^{\alpha-\epsilon} t^{\gamma(\alpha-1)-\epsilon\gamma}.
\end{eqnarray*}
Hence
$$\Pr{{M^\ast_{un/(2B_{n})}} \leq C \frac{{u^{\delta-1}B_{n}}}{t^{1/\alpha_{1}}} } \leq \exp \left( C u^{\delta} t^{\gamma-1/\alpha_{1}} -C u^{\alpha-\epsilon} t^{\gamma(\alpha-1)-\epsilon\gamma}\right).$$
Since $ \alpha-\epsilon>\delta$ and $\gamma(\alpha-1)-\epsilon \gamma>\gamma-1/\alpha_{1}$, the proof is complete.
\end{proof}

\subsection {Bounds on generation sizes}

Our goal is now to establish Theorem \ref{thm:gen}  by using Proposition \ref{prop:Z}, whose proof is postponed to the end of this section. We will use the following simple consequence of Proposition \ref{prop:Z}: For every $ \beta \in (0,\alpha)$, there exists a constant $C >0$ such that
\begin{equation}
\label{eq:Z2}\textrm {for every } n \geq 1, x \geq 0, \qquad \Es { \frac{1}{p_{n}Z^{\ast}_{n}}  \mathbbm{1}_{p_{n}Z^{\ast}_{n} \leq x}} \leq C \cdot x ^ {\beta-1}.
\end{equation}

\begin{proof}[Proof of Theorem \ref{thm:gen}] Fix $ \gamma \in (0, \alpha-1)$  and $\eta>0$. The same size-biasing argument that led us to Lemma \ref{lem:sb}
shows that $$ \Pr{0<Z_{u \frac{n}{B_{n}}}(\t_n)< v B_{n}} \leq    \frac{1}{ \Prmu{ |\tau|=n }} \Es{ \frac{\mathbbm{1}_{Z^{\ast}_{un/(2B_{n})} \leq v B_{n}}}{Z^{\ast}_{un/(2B_{n})}}} \cdot \Es{ \frac{1}{B'_{M^{\ast}_{un/(2B_{n})}}}
}.$$
Lemma  \ref{lem:pB} combined with \eqref{eq:Z2} shows that the first expectation in the previous product is bounded above by $ Cv^{\gamma}/B_{n}$ (for every $v \geq 0$, $u \geq \eta$ and $n \geq 1$). Also, using  \eqref{eq:BB}, observe that 
$$ \Es{ \frac{1}{B'_{M^{\ast}_{un/(2B_{n})}}}   \mathbbm{1}_{M^{\ast}_{un/(2B_{n})} \geq  u^{\delta-1} B_{n}}} \leq \frac{1}{B'_{ \eta^{\delta-1} B_{n}}} \leq \frac{C}{n}$$
for every $v \geq 0$, $u \geq \eta$ and $n \geq 1$. Hence, by  Lemma \ref{lem:exp}, $ \Es{ {1}/{B'_{M^{\ast}_{un/(2B_{n})}}}
} \leq C/n$  for every $v \geq 0$, $u \geq \eta$ and $n \geq 1$ (actually Lemma \ref{lem:exp} is stated for $u \geq 2$, but the arguments carry through). The conclusion the follows since $ n B_{n}\Prmu{ |\tau|=n }$ converges to a positive limit as $n \rightarrow  \infty$.
\end{proof}

By adapting the arguments of this proof in order to control the estimates for small values of $u$, it is possible to establish an upper bound for $\Pr{0<Z_{u \frac{n}{B_{n}}}(\t_n)< v B_{n}} $ valid for every $u,v \geq 0$. For brevity, we shall not enter such considerations.

\begin{proof}[Proof of Corollary \ref{cor:gen}]
For the first assertion, note that $$ \Pr{Z_{u \frac{n}{B_{n}}}(\t_n)> v B_{n}} \leq \Pr{H(\t_{n}) \geq u n/B_{n}},  \qquad  \Pr{Z_{u \frac{n}{B_{n}}}(\t_n)> v B_{n}} \leq \Pr{W(\t_{n}) \geq v B_{n}}.$$
Hence $$  \Pr{Z_{u \frac{n}{B_{n}}}(\t_n)> v B_{n}} \leq \sqrt{\Pr{H(\t_{n}) \geq u n/B_{n}} \cdot  \Pr{W(\t_{n}) \geq v  B_{n}}}.$$
The desired result then follows from Theorems \ref{thm:width} and \ref{thm:height}.

The second assertion is established similarly, by combining Theorem \ref{thm:gen} with the observation that
$$ \Pr{0<Z_{u \frac{n}{B_{n}}}(\t_n)< v B_{n}} \leq\Pr{H(\t_{n}) \geq u n/B_{n}}.$$
This completes the proof.
\end{proof}

\begin{rem}\label{rem:CK}There is an analog of Proposition \ref{prop:Z} for $Z_{n}$ instead of $Z^{\ast}_{n}$. In the proof of Proposition 2.6 in \cite{CK08}, Croydon \& Kumagai show that for every  $ \beta \in (0,\alpha-1)$, there exists a constant $C >0$ such that $\Esmu{ p_{n} Z_{n} \leq x \ \big| \ Z_{n}>0} \leq x^{\beta}$ for every $n \geq 1, x \geq 0$. In this case,  the exponent $\alpha-1$ is  optimal. Indeed, by \cite[Theorem 1]{Sla68}, 
\begin{equation}
\label{eq:cvZ0}\Esmu{e^{-\lambda p_{n}Z_{n}} \ \big| \ Z_{n}>0}  \quad \mathop{\longrightarrow}_{n \rightarrow \infty} \quad  1- \frac{\lambda}{ \left( 1+ \lambda^{\alpha-1} \right)^{ \frac{1}{\alpha-1}}}, \qquad \lambda \geq 0.
\end{equation}
As before, if $Z$ is a random variable having this Laplace transform,  for every $\epsilon>0$, there exists $C>0$ such that $ \Pr{Z \leq x} \geq  C x^{\alpha-1+\epsilon}$ for every $0 \leq x \leq 1$. In particular, combined with Lemma \ref{lem:pB}, this shows that the exponent $\alpha-1$ is optimal in Theorem \ref{thm:gen} when $u$ takes there values in a compact subset of $(0,\infty)$.
\end{rem}

We finally establish Proposition \ref{prop:Z}.

\begin {proof}[Proof of Proposition \ref{prop:Z}] Note that $p_{n}= \Prmu{Z_{n}>0}$.
It is clear that we may assume that $x \geq p_{n}$. In turn, it is sufficient to check the existence of $ \lambda_{0}>0$, $N=N(\lambda_{0})$ and a constant $C=C(\lambda_{0})>0$ such that
\begin{equation}
\label{eq:lapl} \Es{e^{-\lambda p_{n} Z_{n}^{\ast}}} \quad \leq \quad  \frac{C}{\lambda^{\beta}}
\end{equation}
holds for every $ \lambda_{0} \leq  \lambda \leq 1/p_{n}$ and $n \geq N$.
Indeed, if \eqref{eq:lapl} holds, then we have, for every $x \in [p_{n},1/\lambda_{0}]$ and $n \geq N$,
$$ \Pr {p_{n}Z^{\ast}_{n} \leq x} = \Pr {e^{-p_{n} Z^{\ast}_{n}/x} \geq 1} \leq \Es {e^{-p_{n} Z^{\ast}_{n}/x}} \leq C x^{\beta}.$$
From now on, we assume that $ 0 \leq  \lambda \leq 1/p_{n}$.

By definition of $Z^{\ast}_{n}$,  if $f_{m}(s)= \Esmu{s^{Z_{m}}}$, we have $\Es{s^{Z^{\ast}_{n}}}=sf'_{n}(s)=s \prod_{i=0}^{n-1}f_{1}'(f_{i}(s))$. In particular, 
\begin{equation}
\label{eq:sb} \Es {e^{-\lambda p_{n} Z^{\ast}_{n}}}= e^{-\lambda p_{n}} \prod_{i=0}^{n-1} \widehat{F} \left( \Esmu {e^{-\lambda p_{n} Z_{i}}} \right)  \leq  \exp \left( \int_{2/n}^{1} du \ n \ln \left( \widehat{F} \left( \Esmu {e^{-\lambda p_{n} Z_{\fl{un}}}} \right) \right)  \right),
\end{equation}
where we set $ \widehat{F}(s)= \sum_{i=1}^{\infty} i \mu(i) s^{i-1}$.

Fix $ \epsilon \in (0,1)$ and choose $ \lambda_{0}>0$ such that $1- {\lambda_{0}}{(1+ \lambda_{0}^{\alpha-1})^{-1/(\alpha-1)}}< \epsilon$.  By \eqref{eq:cvZ0} and since $\lambda \mapsto \Esmu{e^{-\lambda p_{n}Z_{n}} \ \big| \ Z_{n}>0}$ is decreasing, there exists $N=N(\lambda_{0})$ such that
\begin{equation}
\label{eq:lapl2}\textrm {for every } n \geq N,  \quad \lambda \geq \lambda_{0}, \qquad  \Esmu{e^{-\lambda p_{n}Z_{n}} \ \big| \ Z_{n}>0} \leq \epsilon.
\end{equation}
From now on, we assume that $ \lambda_{0} \leq \lambda \leq 1/p_{n} \leq C n^{ \frac{1}{\alpha-1}+\epsilon}$, where the last inequality follows from  \eqref{eq:equivp} and the Potter bounds. To simplify notation, we set $ \gamma= 1/\left(\frac{1}{\alpha-1}+2\epsilon \right)$. In particular, notice that $ n/\lambda^{\gamma} \rightarrow \infty$ as $ n \rightarrow \infty$. Indeed, since $ \lambda \leq C n^{ \frac{1}{\alpha-1}+\e}$, we have
\begin{equation}
\label{eq:bor2} \frac{n}{ \lambda  ^{ \gamma}} \geq C^{-\gamma}  \cdot {n} \cdot {n^{ - \gamma \left( \frac{1}{\alpha-1}+\epsilon \right) } }   \quad\mathop{\longrightarrow}_{n \rightarrow \infty} \quad  \infty
\end{equation}
because $\gamma \left( {1}/(\alpha-1)+\epsilon \right)<1$.

Again by  \eqref{eq:equivp} and the Potter bounds, there exist two constants a constant $C_{1}=C_{1}(\epsilon), C_{2}=C_{2}(\epsilon)>0$ such that $$ \lambda \frac{p_{n}}{p_{\fl {un}}} \geq C_{1}  \lambda   \left(  \frac{\fl {un}}{n} \right) ^{ \frac{1}{\alpha-1}+2\epsilon } \geq  C_{2} { \lambda}{u^{\frac{1}{\alpha-1}+2\epsilon}}.$$
for every $2/n \leq u \leq 1$ and $n \geq 2$. Set $C_{3}=(\lambda_{0}/C_{2})^\gamma$, so that in particular by \eqref{eq:lapl2}, for every $n$ sufficiently large,
$$u \geq   \frac{C_{3}}{\lambda^{\gamma}} \quad \Longrightarrow \quad  \lambda \frac{p_{n}}{p_{\fl {un}}} \geq \lambda_{0}  \quad \Longrightarrow \quad  \Esmu { \left. e^{-  \lambda\frac{p_{n}}{p_{\fl {un}}} \cdot p_{\fl {un}} Z_{\fl {un}}}  \right | Z_{\fl {un}}>0} \leq \epsilon. 
$$
Thus,  by observing that $\Esmu {e^{-c Z_{i}}}=1- \Prmu{Z_{i}>0}+\Prmu{Z_{i} > 0} \cdot \Esmu {e^{-  c Z_{i}} | Z_{i}>0}$ for every $c>0$ and $i \geq 0$,  for every $n$ sufficiently large and for every $ u \in [C_{3} \lambda^{-\gamma},1]$, we get that 
\begin{equation}
\label{eq:bor3}\Esmu {e^{-\lambda p_{n} Z_{\fl {un}}}}= 1-p_{\fl {un}}  \left(  1-\Esmu {\left. e^{-  \lambda\frac{p_{n}}{p_{\fl {un}}} \cdot p_{\fl {un}} Z_{\fl {un}}}  \right | Z_{\fl {un}}>0} \right) \leq 1-p_{\fl {un}} (1-\epsilon).
\end{equation}
For every $n$ sufficiently large, we have $C_{3} \lambda^{-\gamma} \geq 2/n$ by \eqref{eq:bor2}, so that \eqref{eq:sb} and  \eqref{eq:bor3} yield
\begin{equation}
\label{eq:bor4} \Es {e^{-\lambda p_{n} Z^{\ast}_{n}}} \leq  \exp \left( \int_{C_{3} \lambda^{-\gamma}}^{1} du \ n \ln \left( \widehat{F} \left(  1-p_{\fl {un}} (1-\epsilon)) \right) \right)  \right).
\end{equation}

Now, if $X^{\ast}$ is a random variable with distribution given by $\Pr{X^{\ast}=i}=i\mu(i)$ for $i \geq 0$, we have $ \widehat{F}(s)= \Es{s^{X^{\ast}-1}}$. Hence, by \eqref{eq:equiv3},
$$1-\widehat{F}(1-s)  \quad\mathop{\sim}_{s \downarrow 0} \quad \frac{\Gamma(3-\alpha)}{\alpha-1} \cdot s^{\alpha-1} L( 1/s),$$
so that
$$\ln \left( \widehat{F} \left( 1- s \right) \right)   \quad\mathop{\sim}_{s \downarrow 0} \quad -  \frac{\Gamma(3-\alpha)}{\alpha-1} \cdot s^{\alpha-1} L( 1/s).
$$
In particular,  we may choose $ \eta>0$ sufficiently small in such a way that
$$s \in (0, \eta) \quad \Longrightarrow \quad \ln \left( \widehat{F} \left( 1- s \right) \right) \leq  -(1-\epsilon)  \frac{\Gamma(3-\alpha)}{\alpha-1} \cdot s^{\alpha-1} L( 1/s).$$
 For every $n$ sufficiently large and for every $ u \in [C_{3} \lambda^{-\gamma},1]$, we have $ p_{\fl{un}} \leq  \eta$ by \eqref{eq:bor2}, so \eqref{eq:bor4} implies that 
$$\Es {e^{-\lambda p_{n} Z^{\ast}_{n}}} \leq   \exp \left(  -(1-\epsilon)  \frac{\Gamma(3-\alpha)}{\alpha-1} \cdot  \int_{C_{3} \lambda^{-\gamma}}^{1} du \ n   \left( p_{\fl {un}} (1-\epsilon) \right) ^{\alpha-1} L  \left(  \left(  p_{\fl {un}} (1-\epsilon) \right) ^{-1} \right)  \right).$$
By \eqref{eq:equivp}, we have   $$ L  \left(  \left(  p_{\fl {un}} (1-\epsilon) \right) ^{-1} \right)   \quad\mathop{\sim}_{n \rightarrow \infty} \quad  \frac{\alpha}{\Gamma(3-\alpha) \fl {un} p_{ \fl {un}}^{\alpha-1}},$$
and by \eqref{eq:bor2} this estimate is uniform in $C_{3}\lambda^{-\gamma} \leq u \leq 1$. Hence, for every $n$ sufficiently large, 
\begin{eqnarray*}
\Es {e^{-\lambda p_{n} Z^{\ast}_{n}}} & \leq &   \exp \left(  - (1-\epsilon)^{\alpha+1} \frac{\alpha}{\alpha-1} \cdot  \int_{C_{3} \lambda^{-\gamma}}^{1} du  \frac{n}{\fl{un}}  \right)  \\
& \leq & \exp \left(  - (1- \e)^{\alpha+1}  \frac{\alpha}{\alpha-1} \cdot  \int_{C_{3} \lambda^{-\gamma}}^{1} du \  \frac{1}{u}  \right)  \\
&=&  \exp \left(   (1- \e)^{ \alpha+1}  \frac{\alpha}{\alpha-1} \cdot  \ln \left( C_{3} \lambda^{-\gamma} \right) \right)  = {C_{4}} \cdot {\lambda^{- (1- \e)^{ \alpha+1}   \cdot  \frac{\alpha}{\alpha-1} \cdot \gamma}}\end{eqnarray*}
with $C_{4}=  \exp \left(   (1- \e)^{ \alpha+1}  \frac{\alpha}{\alpha-1} \cdot  \ln( C_{3}) \right) $.
Finally, observe that
$$(1- \e)^{ \alpha+1}   \cdot  \frac{\alpha}{\alpha-1} \cdot \gamma=  \alpha \cdot \frac{(1- \e)^{ \alpha+1}  }{1+ 2 \epsilon(\alpha-1)}.$$
 This completes the proof of \eqref{eq:lapl}, since by choosing $ \e \in (0,1)$ small enough, the quantity $((1- \e)^{ \alpha+1}  )/(1+ 2 \epsilon(\alpha-1))$ will be as close to $1$ as desired.
 \end {proof}

\section{Tail estimates for the stable excursion}

Recall that the tail behavior of the supremum of the associated height process $\He$ has been obtained by Duquesne \& Wang \cite{DW15}, see \eqref{eq:Hinfini} and \eqref{eq:Hzero}. Recall also that we evaluated the asymptotic behavior of the supremum of the stable bridge (Corollary \ref{cor:Xbr}). Here we gather several open questions concerning tail estimates for statistics of the stable bridge $X^{\mathrm{br}}$ and stable excursion $\X$  which have appeared throughout the text:

What is the asymptotic behavior of $\Pr{\sup \Xbr \leq 1/u}$,  $\Pr{\sup \X \geq u}$, $\Pr{ \sup \X \leq 1/u}$, $\Pr { \Delta(X^{*}) \geq u}$,  and $\Pr { \Delta(X^{*}) \leq  1/u}$ as $u \rightarrow  \infty$? What are the values $\Es{(\sup \Xbr)^{p}}$, $ \Es{ (\sup \X)^{p}}$ and $\Es{ \Delta(X^{*})^{p}}$ for $p \geq 1$? In the case $\alpha=2$, $\X$ and $\He$ are multiples of the normalized Brownian excursion, and such estimates are well known (see e.g.~Eq.~(5) and Section 1.1 in \cite{ADS13}).

There also seems to be a duality between the behavior of $\sup \X$ at $\infty$ (resp.~$0$) and the behavior of $\sup \He$ at $0$ (resp.~$\infty$): indeed, the exponent governing the exponential decay should by $\alpha/(\alpha-1)$ for $ \Pr{\sup \X \geq u}$ and $ \Pr{\sup \He \leq 1/u}$, and should by $\alpha$ for  $ \Pr{\sup \X \leq  1/u}$ and $ \Pr{\sup \He \geq  u}$. Can this be seen in a simple way directly in the continuous world?

\section {Appendix}

In this appendix, we prove several useful results concerning the asymptotic behavior of Laplace transforms of critical offspring distributions belonging to domains of attractions of stable laws and of their associated size-biased distributions. As before, assume that $\mu$ is a critical offspring distribution belonging to the domain of attraction of a stable law of index $\alpha \in (1,2]$. Let $\sigma^{2} \in (0,\infty]$ be the variance of $\mu$ and let $X$ is a random variable with law $\mu$. Recall from the Introduction that $L$ is a slowly varying function such that $\textrm{Var}(X \cdot \mathbb{1}_{X \leq n})= n^{2-\alpha} L(n)$. Note that $L(n) \rightarrow \sigma^{2}$ when $\sigma^{2}<\infty$, and that $L(n) \rightarrow  \infty$ when $\sigma^{2}=\infty$ and $\alpha=2$. Hence, if $X$ is a random variable with law $\mu$, we have
$$\Es{X^{2} \mathbbm{1}_{X \leq n}}  \quad \mathop{\sim}_{n \rightarrow \infty} \quad n^{2-\alpha} L(n) +1.$$
since $\Es{X \mathbbm{1}_{X \leq n}} \rightarrow 1$ as $n \rightarrow \infty$.
The term ``$+1$'' is not negligible only when $\sigma^{2}<\infty$ (in which case $\alpha=2$).

\paragraph{Offspring distributions.} Set $G(s)= \sum_{i \geq 0} \mu(i) s^{i}$ for $0 \leq s \leq 1$. 
Then by e.g.~\cite[Lemma 4.7]{BS15}
\begin{equation}
\label{eq:equiv0}G(s)-s  \quad\mathop{\sim}_{s \uparrow 1} \quad   \frac{\Gamma(3-\alpha)}{\alpha(\alpha-1)} \cdot (1-s)^{\alpha} L( (1-s)^{-1}).
\end{equation}
We stress that this holds in the both cases $\sigma^{2}<\infty$ and $\sigma^{2}=\infty$.

Also, if $W$ is a random variable with distribution $ \Pr {W=i}= \mu(i+1)$ for $i \geq -1$, since $ \Es {e^{-\lambda W}}= e^{\lambda}G(e^{-\lambda})$ for $\lambda>0$, we have
\begin{equation}
\label{eq:equivalents}\Es {e^{-\lambda W}} -1  \quad\mathop{\sim}_{\lambda \downarrow 0} \quad   \frac{\Gamma(3-\alpha)}{\alpha(\alpha-1)} \cdot  \lambda^{\alpha} L(1/\lambda).
\end{equation}
Again, this holds  in the both cases $\sigma^{2}<\infty$ and $\sigma^{2}=\infty$.

\paragraph{Size-biased offspring distributions.} Let $ {\mu}^{\ast}$ be the so-called size-biased probability distribution on $ \Z_{+}$ defined by $ {\mu}^{\ast}(i)=i\mu(i)$ for $i \geq 0$.  Note that $  {\mu}^{\ast}$ is indeed a probability distribution since $\mu$ is critical. Let $X^{\ast}$ be a random variable having law $\mu^{\ast}$. When $\mu$ has finite variance, we claim that 
\begin{equation}
\label{eq:equivfini} 1-\Es{s^{X^{\ast}}}  \quad\mathop{\sim}_{s \uparrow 1} \quad (1-s)(\sigma^{2}+1),
\end{equation}
and when $\mu$ has infinite variance, we claim that
\begin{equation}
\label{eq:equivalents2}1-\Es{s^{X^{\ast}}}  \,\, \mathop{\sim}_{s \uparrow 1} \,\,  \frac{\Gamma(3-\alpha)}{\alpha-1} \cdot (1-s)^{\alpha-1} L((1-s)^{-1}), \quad  1- \Es { e^{-\lambda {X}^{\ast}}} \,\,\mathop{\sim}_{\lambda \downarrow 0} \,\,  \frac{\Gamma(3-\alpha)}{\alpha-1} \cdot \lambda^{\alpha-1} L(1/\lambda).
\end{equation}

When $\mu$ has finite variance the claim \eqref{eq:equivfini} simply follows from the fact that $\Es{X^{\ast}}=\sigma^{2}+1$.

Now assume that $\mu$ has infinite variance. Then there exists a slowly varying function $L_{1}$ such that $\Pr{X \geq n}= \mu([n,\infty))=L_{1}(n)/n^{\alpha}$ (see \cite[Corollary XVII.5.2 and (5.16)]{Fel71}) when $\alpha<2$, we have $L_{1}(n)= \frac{2-\alpha}{\alpha} L(n)$, and $L_{1}(n)/L(n) \rightarrow 0$ as $n \rightarrow  \infty$  when $\alpha=2$. As a consequence, $\mu^{\ast}$ belongs to the domain of attraction of a stable law of index $\alpha-1$, because
\begin{equation}
\label{eq:star} {\mu}^{\ast}([n,\infty))  \quad\mathop{\sim}_{n \rightarrow \infty} \quad   \frac{\alpha}{\alpha-1} \cdot \frac{L_{1}(n)}{n^{\alpha-1}}
\end{equation}since we can write $   {\mu}^{\ast}([n,\infty)) = (n-1)  {\mu}([n,\infty)+ \sum_{j=n}^{\infty} \mu([j,\infty))$.

If $\alpha<2$,  \eqref{eq:star} and \cite[Corollary XVII.5.2 and (5.16)]{Fel71} give that 
$$  \Es{(X^\ast)^{2} \mathbbm{1}_{X^{\ast} \leq n}}  \quad \mathop{\sim}_{n \rightarrow \infty} \quad n^{3-\alpha} \cdot  \frac{2-\alpha}{3-\alpha} L(n),$$
and \eqref{eq:equivalents2}  result follows e.g. by \cite[Lemma 4.6]{BS15}.

Now assume that $\alpha=2$ and set $q^{\ast}_{i}= \Pr{X^{\ast} > i}$ for $i \geq 0$. Then
$$\sum_{i=0}^{n}q^{\ast}_{i}= \Es{X^{2} \mathbbm{1}_{X \leq n}}+ (n+1) \mu^{\ast}([n+1,\infty))  \quad \mathop{\sim}_{n \rightarrow \infty} \quad  L(n).$$
Indeed,  we know that $L_{1}(n)/L(n) \rightarrow 0$ as $n \rightarrow  \infty$.  Hence, by \cite[Thm. XIII.5.5]{Fel71}, we have $ \sum_{i=0}^{\infty}q^{\ast}_{i} s^{i} \sim L((1-s)^{-1})$ as $s \uparrow 1$.  Then
$$1-\Es{s^{X^{\ast}}}=(1-s) \sum_{i=0}^{\infty}q^{\ast}_{i} s^{i}  \quad \mathop{\sim}_{s \uparrow 1} \quad  (1-s)L((1-s)^{-1}).$$
Our claim \eqref{eq:equivalents2} then follows by taking $s=e^{-\lambda}$. 

Finally, from \eqref{eq:equivfini}  and \eqref{eq:equivalents2}  it is a simple matter to see that the estimates
\begin{equation}
\label{eq:equiv3}1-\Es{s^{X^{\ast}-1}} \, \mathop{\sim}_{s \uparrow 1} \,  \frac{\Gamma(3-\alpha)}{\alpha-1} \cdot (1-s)^{\alpha-1} L((1-s)^{-1}),  \quad   1- \Es { e^{-\lambda ({X}^{\ast}-1) } } \,\mathop{\sim}_{\lambda \downarrow 0} \,  \frac{\Gamma(3-\alpha)}{\alpha-1} \cdot \lambda^{\alpha-1} L(1/\lambda)
\end{equation}
hold in  both the cases $\sigma^{2}<\infty$ and $\sigma^{2}=\infty$ (when $\mu$ has infinite variance and $\alpha=2$ we use the fact that $L(n) \rightarrow  \infty$ as $n \rightarrow \infty$).

\end{document}